\documentclass[11pt,oneside]{amsart}
\title{On splitting and splittable families}
\author{Samuel Coskey}
\address{Boise State University, Boise, ID, USA}
\email{scosky@gmail.com}
\author{Bryce Frederickson}
\address{Emory University, Atlanta, GA, USA}
\email{bfrede4@emory.edu}
\author{Samuel Mathers}
\address{Chicago, IL, USA}
\email{smathers304@gmail.com}
\author{Hao-Tong Yan}
\address{New York, NY, USA}
\email{haotongyan23@gmail.com}
\usepackage{geometry}
\usepackage{amsmath,amssymb,amsthm}
\usepackage{mathpazo,bm}
\usepackage{setspace}
\usepackage{float}
\usepackage{tikz}
\usepackage{mathrsfs}
\usepackage[hidelinks]{hyperref}
\usepackage{todonotes}

\newcommand{\floor}[1]{\ensuremath{\left\lfloor #1 \right\rfloor}}
\newcommand{\ceiling}[1]{\ensuremath{\left\lceil #1 \right\rceil}}

\DeclareMathOperator{\splitters}{split}

\newtheorem{theorem}{Theorem}[section]
\newtheorem{lemma}[theorem]{Lemma}
\newtheorem{proposition}[theorem]{Proposition}

\theoremstyle{definition}
\newtheorem{definition}[theorem]{Definition}
\newtheorem*{conjecture}{Conjecture}

\usepackage{etoolbox}
\makeatletter\pretocmd{\@seccntformat}{\S}{}{}
  \pretocmd{\@subseccntformat}{\S}{}{}\makeatother

\begin{document}
\maketitle

\begin{abstract}
  A set $A$ is said to split a finite set $B$ if exactly half the elements of $B$ (up to rounding) are contained in $A$. We study the dual notions: (1) splitting family, which is a collection of sets such that any subset of $\{1,\ldots,k\}$ is split by a set in the family, and (2) splittable family, which is a collection of sets such that there is a single set $A$ that splits each set in the family.
  
  We study the minimum size of a splitting family on $\{1,\ldots,k\}$, as well as the structure of splitting families of minimum size. We use a mixture of computational and theoretical techniques. We additionally study the related notions of $\mathord{\leq}4$-splitting families and $4$-splitting families, and we provide lower bounds on the minimum size of such families.
  
  Next we investigate splittable families that are just on the edge of unsplittability in several senses. First, we study splittable families that have the fewest number of splitters. We give a complete characterization in the case of two sets, and computational results in the case of three sets. Second, we define a splitting game, and study splittable families for which a splitter cannot be found under adversarial conditions.
\end{abstract}

\section{Introduction}

This article concerns the dual notions of splitting families and splittable families, which arise naturally in several areas of combinatorics. Both types of families involve the following key notion:

\begin{definition}
  Let $A,B$ be subsets of $[k]=\{1,\ldots,k\}$. We say that $A$ \emph{splits} $B$, or $A$ is a \emph{splitter} of $B$, if $|A \cap B| = |B|/2$ for $|B|$ even, and $|A \cap B| = (|B|\pm1)/2$ for $|B|$ odd.
\end{definition}


A collection $\mathcal A$ of subsets of $[k]$ is said to be a \emph{splitting family} on $k$ if for any $B\subset[k]$ there exists an $A\in\mathcal A$ such that $A$ splits $B$. Splitting families are also called splitting systems, and along with separating systems such families have been used to aid in combinatorial search algorithms \cite{stinson-baby}, see also \cite{stinson-mt,li-hao}. In sections 2--3 below, we will study several questions surrounding splitting families and variations of splitting.

A collection $\mathcal B$ of subsets of $[k]$ is said to be a \emph{splittable family} if there exists $A\subset[k]$ such that for each $B\in\mathcal B$ we have $A$ splits $B$ (we sometimes abbreviate this: $A$ splits $\mathcal B$). Splittable families correspond to the families with combinatorial discrepancy $\leq1$, the optimal value. We refer to \cite{chazelle} as a starting point for discrepancy theory and its numerous applications. In sections 4--5, we will study several questions surronding splittable families and variations of splittability.

Beginning with splitting families, our first investigation focuses on the question of the least size of a splitting family on $k$. Splitting families were used in \cite{stinson-baby} as an aid in an algorithm for solving the low Hamming weight discrete logarithm problem. In that article, the author presents a construction, attributed varyingly to Coppersmith and Galvin, of a splitting family on $k$ with $\lceil{k/2}\rceil$ sets. Specifically he shows that the family consisting of $A_i=\{i,i+1\ldots,i+\ceiling{k/2}-1\}$ for $i=1,\ldots,\ceiling{k/2}$ is a splitting family on $k$. (We will refer to this as a \emph{standard} splitting family.)

The article \cite{stinson-baby} left it unclear whether there exists a splitting family of smaller size than the standard splitting family. It is clear that a smaller splitting family would improve their algorithms, but it is not straightforward to observe whether such an object exists. In Section~2 we affirm that the standard splitting family is optimal in the sense that any splitting family on $k$ has size at least $\ceiling{k/2}$. This result follows from the results of \cite{alon} and was pointed out to us by a colleague \cite{yost-wolff}.

It is also natural to ask whether the standard splitting family is the \emph{unique} splitting family of optimal size. We provide computational results essentially confirming this conjecture for $k\leq 16$ (with just one notable exception). We further show that under additional assumptions, the conjecture holds for all $k$.

In Section~3 we study two broad generalizations of splitting families which arise naturally in the above investigations. A collection $\mathcal A$ is said to be a \emph{$\mathord{\leq}4$-splitting family} on $k$ if for any $B\subset[k]$ such that $|B|\leq 4$ there exists an $A\in\mathcal A$ such that $A$ splits $B$. The notion of $4$-splitting family is defined analogously. We will provide lower bounds on the minimum size of a $\mathord{\leq}4$-splitting family and of a $4$-splitting family. The latter bound corrects a minor error in \cite{stinson-mt}.

Our second major area of investigation concerns the boundary between splittable and unsplittable families. It was shown in \cite{charikar} (and elaborated in \cite{reu16}) that the general problem of deciding whether a given family is splittable is NP-complete. It is then natural to ask which families $\mathcal B$ are ``just splittable,'' that is, splittable but with the fewest number of distinct sets $A$ that split $\mathcal B$. The answer to the question of which $n$-set families $\mathcal B$ are ``just splittable'' can be used to find bounds on the least size of an $n$-splitting family on $[k]$, see \cite{reu14}.

In Section~4 we give results on $n$-set families on $[k]$ with the fewest number of splitters for $n\leq3$. When $n=1$ and $\mathcal B=\{B\}$, the minimum number of splitters occurs when $|B|=k$ or $k-1$, and the number of splitters is asymptotic to $2^k/\sqrt{k}$. When $n=2$ and $\mathcal B=\{B_1,B_2\}$, it was shown in \cite{reu14} that if $|B_i|\leq k/2$ then the minimum number of splitters occurs when $B_1$ and $B_2$ are disjoint. In this article we remove the cardinality constraints $|B_i|\leq k/2$ and find that the minimum number of splitters occurs when $|B_i|\approx 2k/3$ and $|B_1\cap B_2|\approx k/3$. In both the disjoint and general cases, the number of splitters is asymptotic to $2^k/k$. We also calculate an analytical formula for approximating the number of splitters of an arbitrary two-set family $\mathcal B$.

When $n=3$ the situation is somewhat more complex and we provide partial answers and computational results. Based solely on the short pattern $2^k/k^{1/2}$, $2^k/k^1$, the authors of \cite{reu14} conjectured that for $n=3$ the minimum number of splitters is asymptotic to $2^k/k^{3/2}$. Our computational results support this conjecture. However the explanation for the formula $2^k/k^{3/2}$ does not seem to be the same as the explanation for the cases $n=1,2$. Indeed, when $n\geq3$ there exist unsplittable families, unlike when $n<3$, and the ``just splittable'' families are very similar to the unsplittable ones in appearance.

In Section~5, we define and study a strategic variant of the notion of splittability. We define the \emph{splitting game}, in which players Split and Skew are given a family $\mathcal B$ and collaborate to construct a set $A$. Split wins if $A$ splits $\mathcal B$ and Skew wins otherwise. We establish several general lemmas that allow one to simplify the analysis of a given instance of the game. For families $\mathcal B$ of three or fewer sets, we give a complete characterization of when each player has a winning strategy. We also provide complete solutions to several special case studies, such as the tic-tac-toe style game which arises when the splitting game is played on a grid.

\textbf{Acknowledgement.} This article represents a portion of the research carried out during the 2018 REU CAD at Boise State University. The program was supported by NSF award \#DMS-1659872 and by Boise State University.

\section{Splitting families}

Recall from the introduction that a splitting family $\mathcal A$ on $k$ is a collection of sets for which every subset of $[k]$ is split by some set in $\mathcal A$. In this section we investigate several questions surrounding the minimum size of a splitting family on $k$, as well as the structure of families of minimum size.

Recall from the introduction that for even $k$, the \emph{standard splitting family} on $k$ consists of $A_i=\{i,i+1\ldots,i+k/2-1\}$ for $i=1,\ldots,k/2$. For any even $k$, the standard splitting family on $k$ is in fact a splitting family, see for instance \cite{enomoto} and \cite{stinson-baby}. As a consequence, for any $k$ the minimum size of a splitting family on $k$ has upper bound $\ceiling{k/2}$. The next result states that this bound is tight. We are grateful to Calvin Yost--Wolff for discovering the proof and allowing us to include it here.

\begin{theorem}[\cite{yost-wolff}]
  Any splitting family on $k$ has size at least $k/2$. As a consequence, the minimum size of a splitting family on $k$ is exactly $\ceiling{k/2}$.
\end{theorem}

\begin{proof}
  Suppose that $\mathcal A$ is a splitting family on $k$. For each $A\in\mathcal A$ let $\bm{v}_A$ be the $\pm1$ characteristic vector of $A$, that is, $\bm{v}_A(i)=1$ for $i\in A$ and $\bm{v}_A(i)=-1$ for $i\notin A$. Define a polynomial $p(\bm{x})$ in $k$ variables by
  \[p(\bm{x})=\prod_{A\in\mathcal A}\bm{v}_A\cdot\bm{x}\text{.}
  \]
  Then $p$ vanishes on every $0,1$-vector with an even number of $1$'s. Next define
  \[q(\bm{x})=p((x_1+1)/2,\ldots,(x_k+1)/2)\text{.}
  \]
  and note that $q$ vanishes on every $\pm1$-vector with an even number of $1$'s.
  
  Finally let $\bar q$ be the polynomial obtained from $q$ by repeatedly replacing occurrences of $x_i^2$ with $1$ in each monomial of $q$. Then $\bar q$ agrees with $q$ on every $\pm1$-vector, and hence $\bar q$ vanishes on every $\pm1$-vector with an even number of $1$'s.
  
  Now $\bar q$ is multilinear in the sense that it is affine in each coordinate. Moreover $\bar q\neq 0$: indeed, $p$ does not vanish on any $0,1$-vector with an odd number of $1$'s, so $q$ does not vanish on any $\pm1$-vector with an odd number of $1$'s, so $\bar q$ does not vanish on any $\pm1$-vector with an odd number of $1$'s. It follows from \cite[Lemma~2.1]{alon} that $\deg(\bar q)\geq k/2$. We can now conclude that $|\mathcal A|=\deg(p)=\deg(q)\geq\deg(\bar q)\geq k/2$.
\end{proof}

Given this result, it is natural to wonder whether the standard splitting family on $k$ is the unique splitting family of size $\ceiling{k/2}$. In the rest of this section, we present partial answers to this question.

We begin with computational results. In the following we say that a splitting family $\mathcal A$ on $k$ is \emph{uniform} if every set $A\in\mathcal A$ has size $|A|=\floor{k/2}$ or $\ceiling{k/2}$. Observe that the standard splitting family on $k$ is uniform. We also say that splitting families $\mathcal A,\mathcal A'$ are \emph{equivalent} if $\mathcal A'$ can be obtained from $\mathcal A$ by complementing sets in $\mathcal A$ and permuting the elements of $[k]$.

\begin{proposition}
  \begin{itemize}
    \item For all $k\leq16$, every minimal splitting family is uniform.
    \item Among even $k\leq16$, there is just one minimal splitting family which is not equivalent to a standard splitting family, namely $k=8$ and
    \[\mathcal F=\{\{1,2,3,4\},\{1,2,5,6\},\{3,4,5,6\},\{1,3,5,7\}\}\text{.}
    \]
  \end{itemize}
\end{proposition}

The method of search centered on the following notion. We say that splitting family $\mathcal A=\{A_1, \ldots, A_n\}$ on $k$ is \emph{extendable} if it is the restriction to $k$ of some splitting family on $k+1$. Our algorithm performed an exhaustive search of non-extendable splitting families with a fixed number of sets. We do not see any immediate generalizations to the exceptional splitting family with $k=8$. The code and its output may be found in \cite{bryce}.

In our next results we identify the structure of minimal splitting families under a strong additional hypothesis. To begin, we recall that if $\mathcal A=\{A_1,\ldots,A_n\}$ is a family of subsets of $[k]$, its \emph{incidence matrix} is the $n\times k$ matrix $A$ with a $1$ in the $(i,j)$ entry if $j\in A_i$, and a $0$ otherwise. When every pair from $[k]$ is split by some set in $\mathcal A$, the columns of $A$ are distinct elements of the Hamming cube $\{0,1\}^n$. In this case, the \emph{Hamming representation} of $\mathcal A$ is the graph on the subset $S\subset\{0,1\}^n$ consisting of the columns of $A$, with adjacency defined by $s\sim t$ if and only if $s,t$ differ in exactly one entry.

If $\bar x$ is a subset of the indices $1,\ldots,n$ we let $\delta_{\bar x}$ denote the element of $\{0,1\}^n$ with $1$'s in the entries indicated by $\bar x$ and $0$'s everywhere else. We say that a set $B_\delta$ of such vectors is \emph{split} if some index appears in exactly half (up to rounding) of the elements of $B_\delta$ with this representation. Thus, given a family $\mathcal A$ with Hamming representation $S$ on $\{\delta_{\bar x_1}, \ldots, \delta_{\bar x_k}\}$, we can identify a subset $B \subset [k]$ with the subset $B_\delta = \{\delta_{\bar x_j} \mid j \in B\}$ of $S$, which is split if and only if some set in $\mathcal A$ splits $B$. In our graphical representations we stratify the elements $\delta_{\bar x}$ by their \emph{weight}, $|\bar x|$. See Figure~\ref{fig:hamming} for two examples of Hamming representations.

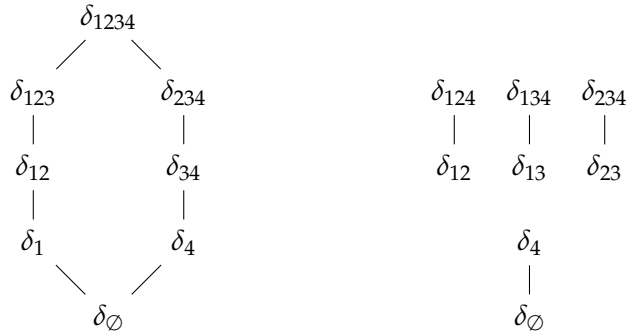
\begin{figure}[ht]
  \centering
  \begin{tikzpicture}
    \node(0) at (0,0) {$\delta_\emptyset$};
    \node(1) at (-1,1) {$\delta_1$};
    \node(4) at (1,1) {$\delta_4$};
    \node(12) at (-1,2) {$\delta_{12}$};
    \node(34) at (1,2) {$\delta_{34}$};
    \node(123) at (-1,3) {$\delta_{123}$};
    \node(234) at (1,3) {$\delta_{234}$};
    \node(1234) at (0,4) {$\delta_{1234}$};
    \draw (0)--(1)--(12)--(123)--(1234)--(234)--(34)--(4)--(0);
  \end{tikzpicture}
  \hspace{1in}
  \begin{tikzpicture}
    \node(0) at (0,0) {$\delta_\emptyset$};
    \node(4) at (0,1) {$\delta_4$};
    \node(12) at (-1,2) {$\delta_{12}$};
    \node(13) at (0,2) {$\delta_{13}$};
    \node(23) at (1,2) {$\delta_{23}$};
    \node(124) at (-1,3) {$\delta_{124}$};
    \node(134) at (0,3) {$\delta_{134}$};
    \node(234) at (1,3) {$\delta_{234}$};
    \draw (0)--(4);
    \draw (12)--(124);
    \draw (13)--(134);
    \draw (23)--(234);
  \end{tikzpicture}
  \caption{Left: The Hamming representation of the standard splitting family on $k=8$. Right: The Hamming representation of the nonstandard minimal splitting family on $k=8$.\label{fig:hamming}}
\end{figure}

In the following result, we say that a family $\mathcal A$ of subsets of $[k]$ is a \emph{$\mathord\leq t$-splitting family} if every subset of $[k]$ of size at most $t$ is split by some set in $\mathcal A$. We say that a $\mathord\leq 2$-splitting family $\mathcal A$ is \emph{connected} if its Hamming representation is a connected graph.

\begin{theorem}
  \label{theorem:connected}
  Let $\mathcal A$ be a connected $\mathord\leq4$-splitting family on $k$ of minimum size. Then for $k$ even, $\mathcal A$ is equivalent to the standard splitting family on $k$, and for $k$ odd, $\mathcal A$ is equivalent to the standard splitting family on $k+1$ with the point $k+1$ removed. In particular, $|\mathcal A|=\ceiling{k/2}$.
\end{theorem}

The proof consists of a series of lemmas. In the following, when $\bar x$ and $\bar y$ are sets of indices, we let $\bar x\bar y$ denote the union of the two sets.

\begin{lemma}[Y lemma]
  \label{lemma:Y}
  Let $S$ be the Hamming representation of a $\mathord\leq4$-splitting family $\mathcal A$. Then every element of $S$ has degree at most $2$ in $S$. In fact, $S$ cannot contain an arrangement of the following four ``Y'' types, where $\bar x, \bar y, \bar w, \bar u, \bar v$ are pairwise disjoint, and $\bar u, \bar v, \bar w$ are nonempty:
  \begin{enumerate}
    \item $\{\delta_{\bar x},\delta_{\bar x\bar y\bar u},\delta_{\bar x\bar y\bar v},\delta_{\bar x\bar y\bar w}\}$;
    \item $\{\delta_{\bar x},\delta_{\bar x\bar u},\delta_{\bar x\bar u\bar v},\delta_{\bar x\bar u\bar w}\}$;
    \item $\{\delta_{\bar x \bar u},\delta_{\bar x \bar v},\delta_{\bar x\bar u \bar v},\delta_{\bar x\bar u\bar v \bar w}\}$;
    \item $\{\delta_{\bar x \bar u \bar v},\delta_{\bar x\bar u \bar w},\delta_{\bar x \bar v \bar w},\delta_{\bar x\bar u\bar v \bar w}\}$.
  \end{enumerate}
\end{lemma}

These four arrangements are shown in Figure~\ref{fig:4forbidden}.

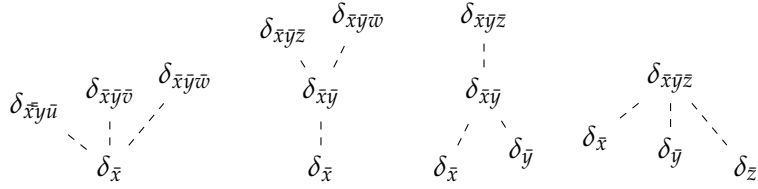
\begin{figure}[ht]
  \centering
  \begin{tikzpicture}
    \node(x) at (0,0) {$\delta_{\bar x}$};
    \node(xyu) at (-1,0.8) {$\delta_{\bar x\bar y\bar u}$};
    \node(xyv) at (0,1) {$\delta_{\bar x\bar y\bar v}$};
    \node(xyw) at (1,1.2) {$\delta_{\bar x\bar y\bar w}$};
    \draw[dashed] (x) edge(xyu) edge(xyv) edge(xyw);
  \end{tikzpicture}\quad
  \begin{tikzpicture}
    \node(x) at (0,0) {$\delta_{\bar x}$};
    \node(xu) at (0,1) {$\delta_{\bar x\bar u}$};
    \node(xuv) at (-.5,1.8) {$\delta_{\bar x\bar u\bar v}$};
    \node(xuw) at (.5,2) {$\delta_{\bar x\bar u\bar w}$};
    \draw[dashed] (x) edge(xu);
    \draw[dashed] (xu) edge(xuv) edge(xuw);
  \end{tikzpicture}\quad
  \begin{tikzpicture}
    \node(xu) at (-.5,-1) {$\delta_{\bar x \bar u}$};
    \node(xv) at (.5,-.8) {$\delta_{\bar x \bar v}$};
    \node(xuv) at (0,0) {$\delta_{\bar x\bar u \bar v}$};
    \node(xuvw) at (0,1) {$\delta_{\bar x\bar u\bar v\bar w}$};
    \draw[dashed] (xu) edge(xuv)  (xv) edge(xuv)  (xuv) edge(xuvw);
  \end{tikzpicture}\quad
  \begin{tikzpicture}
    \node(xuvw) at (0,0) {$\delta_{\bar x\bar u\bar v\bar w}$};
    \node(xuv) at (-1,-0.8) {$\delta_{\bar x\bar u \bar v}$};
    \node(xuw) at (0,-1) {$\delta_{\bar x \bar u \bar w}$};
    \node(xvw) at (1,-1.2) {$\delta_{\bar x\bar v \bar w}$};
    \draw[dashed] (xuvw) edge(xuv) edge(xuw) edge(xvw);
  \end{tikzpicture}
  \caption{Forbidden ``Y'' arrangements for the Hamming representation of a $\leq4$-splitting family; (a)--(d) appear left--right.\label{fig:4forbidden}}
\end{figure}

\begin{proof}
  To prove the arrangements of types (a)--(d) are impossible, it is sufficient to observe that in each type (a)--(d), no index appears exactly twice among the subscripts, as this means that the arrangement is not split.

  Now this implies every element of $S$ has degree at most $2$. To see this, suppose towards a contradiction that some element $s\in S$ is adjacent in the Hamming hypercube to three other elements $s_1,s_2,s_3\in S$. We claim that the $4$-element set $\{s,s_1,s_2,s_3\}$ is of one of the types~(a)--(d). Indeed, each of $s_1,s_2,s_3$ has weight one higher or one lower than $s$. If all three are higher the set is of type~(a), if two of the three are higher the set is of type~(b), if two of the three are lower the set is of type~(c), and if all three are lower the set is of type~(d). In all cases we obtain a contradiction to the previous paragraph.
\end{proof}

We remark that avoiding the types~(a)--(d) does not guarantee a $\mathord\leq4$-splitting family. For example $\{\delta_1,\delta_2,\delta_3,\delta_4\}$ is not split and not of any of the types~(a)--(d).

\begin{lemma}
  \label{lem:cycle}
  Let $S$ be the Hamming representation of a $\mathord\leq4$-splitting family $\mathcal A$ with $\emptyset\notin\mathcal A$. If $S$ is a simple cycle, then $\mathcal A$ is equivalent to the standard splitting family on $k=|S|$.
\end{lemma}

\begin{proof}
  By complementing sets of $\mathcal A$, we can assume $\delta_\emptyset\in S$. We first argue that the stratified graph of $S$ has the shape of a standard arrangement, that is, $S$ has a single element of maximum weight, and two elements at each intermediate weight (see Figure~\ref{fig:hamming} left). Indeed otherwise $S$ must contain elements of the form $\delta_{\bar xa},\delta_{\bar x},\delta_{\bar xb}$ (informally, the cycle must have a local minimum at a location other than $\delta_\emptyset$). But then the set $\{\delta_\emptyset,\delta_{\bar x},\delta_{\bar xa},\delta_{\bar xb}\}$ is a Y of type~(b), a contradiction.
  
  Next we argue that the $\delta$-labeling of the elements of $S$ is equivalent to a standard arrangement up to renaming the sets. To begin we know the bottom element (least weight) is $\delta_\emptyset$ and we can say that the top element (greatest weight) is $\delta_{12\cdots n}$ where $n$ is the total number of sets. Furthermore we can assume without loss of generality that the labels up the left-hand side are $\delta_{1},\delta_{12},\ldots,\delta_{12\cdots n-1}$.
  
  
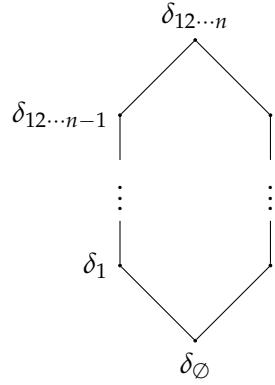
\begin{center}
    \begin{tikzpicture}[n/.style={draw,circle,fill,inner sep=0}]
      \node[n,label=below:$\delta_\emptyset$](0) at (0,0) {};
      \node[n,label=left:$\delta_1$](1) at (-1,1) {};
      \node[n](4) at (1,1) {};
      \node(12) at (-1,2) {$\vdots$};
      \node(34) at (1,2) {$\vdots$};
      \node[n,label=left:$\delta_{12\cdots n-1}$](123) at (-1,3) {};
      \node[n](234) at (1,3) {};
      \node[n][label=above:$\delta_{12\cdots n}$](1234) at (0,4) {};
      \draw (0)--(1)--(12)--(123)--(1234)--(234)--(34)--(4)--(0);
    \end{tikzpicture}
  \end{center}

  To conclude, we need only show that the labels down the right-hand side are exactly $\delta_{2\cdots n},\delta_{3\cdots n},\ldots,\delta_{n}$. To see this, assume as an inductive hypothesis that the first $i-1$ vertices down the right-hand side are $\delta_{2\cdots n},\ldots,\delta_{i\cdots n}$, for some $i\geq1$.
  
  We claim that the remaining vertices down the right-hand side must all omit $i$. Otherwise we would be able to find two distinct such vertices of the form $\{\delta_{\bar x},\delta_{\bar xi}\}$. But then the set $\{\delta_{1\cdots i},\delta_{\bar x},\delta_{\bar xi},\delta_{i\cdots n}\}$ is not split. This proves the claim.
  
  Now we can conclude that the vertex immediately below $\delta_{i\cdots n}$ must be $\delta_{(i+1)\cdots n}$. This completes the inductive step, and the proof.
\end{proof}

In the next result, if $S$ is a Hamming representation of an $n$-set $\mathord\leq4$-splitting family $\mathcal A$, we say that $S$ is \emph{maximal} if whenever an element of $\{0,1\}^n$ is added to $S$ the result is no longer a Hamming representation of a $\mathord\leq4$-splitting family.

\begin{lemma}
  Let $S$ be the Hamming representation of a $\mathord\leq4$-splitting family $\mathcal A$ with $\emptyset\notin\mathcal A$. If $S$ is a simple cycle, then $S$ is maximal.
\end{lemma}

\begin{proof}
  By the previous lemma, it is sufficient to show that the standard arrangement $S=\{\delta_\emptyset,\delta_1,\ldots,\delta_{1\cdots n},\cdots,\delta_n\}$ is maximal. Now suppose towards a contradiction that there exists $\delta_{\bar x}\notin S$ such that $S\cup\{\delta_{\bar x}\}$ is a Hamming representation of a $\mathord\leq4$-splitting family.
  
  We first claim that either $1\in\bar x$ or $n\in\bar x$. Indeed otherwise we would have that the set $\{\delta_\emptyset,\delta_1,\delta_n,\delta_{\bar x}\}$ is a $Y$ of type~(a). Thus we can suppose without loss of generality that $1\in\bar x$.
  
  
  Now let $i\in 2,\ldots,n$ be the least index such that $i\notin\bar x$. Then the set $\{\delta_{1\cdots i-2},\delta_{1\cdots i-1},\delta_{1\cdots i},\delta_{\bar x}\}$ is a $Y$ of type~(b), a contradiction which completes the proof.
\end{proof}

The lemmas so far imply that if $S$ is the Hamming representation of a $\mathord\leq4$-splitting family, then $S$ is either a single simple cycle or else a disjoint union of paths. Next we investigate constraints on the paths.

\begin{lemma}
  Let $S$ be the Hamming representation of a $\mathord\leq4$-splitting family $\mathcal A$ with $\emptyset\notin\mathcal A$. If $S$ is a path, then $\mathcal A$ is equivalent to the restriction to $k=|S|$ of the standard splitting family on $k'=2|\mathcal A|$.
\end{lemma}

\begin{proof}
  By complementing some of the sets of $\mathcal A$, we can suppose without loss of generality that one end of the path $S$ is $\delta_\emptyset$. We will show that $S$ is a subset of the Hamming representation of a standard splitting family. By permuting the indices of the sets, we can suppose that the longest initial segment of increasing weights is $\delta_\emptyset,\delta_1,\ldots,\delta_{1\cdots n}$ for some $n$. If there are no more elements of $S$, then we are done.
  
  Otherwise we can suppose that $S$ turns after $\delta_{1\cdots n}$. Then this must be the last turn, since an additional turn would form a vee shape, which together with $\delta_\emptyset$ would form a $Y$ of type~(b). Now the inductive argument from Lemma~\ref{lem:cycle} implies that the remaining points of $S$ after $\delta_{1\cdots n}$ must be $\delta_{2\cdots n},\ldots,\delta_{i\cdots n}$. It follows that $n=|\mathcal A|$, and that $S$ is the restriction of the standard splitting family on $n$.
\end{proof}

We are now ready to conclude the proof of the theorem.

\begin{proof}[Proof of Theorem~\ref{theorem:connected}]
  The lemmas together imply that $\mathcal A$ is either equivalent to the standard splitting family on $k$ or the restriction to $k$ of the standard splitting family on $k'=2|\mathcal A|$. Since $\mathcal A$ is of minimum size, $|\mathcal A|\leq\ceiling{k/2}$, which implies that $k'=k$ or $k'=k+1$ (if $k$ is odd).
\end{proof}

\section{Splitting subsets of size $2$ and $4$}

Recall from the previous section that $\mathcal A$ is a $\mathord{\leq}t$-splitting family on $k$ if $\mathcal A$ splits all subsets of $[k]$ of size at most $t$. We will also say that $\mathcal A$ is a \emph{$t$-splitting family} on $k$ if $\mathcal A$ splits all subsets of $[k]$ of size exactly $t$. In \cite{stinson-baby}, $t$-splitting families are used in an algorithm to solve the weight $t$ discrete log problem. In this section we study the least size of a $4$-splitting family and the least size of a $\mathord{\leq}4$-splitting family on $k$.

Beginning with $4$-splitting families, the article \cite{ling} established a lower bound on the least size of a \emph{uniform} $4$-splitting family (that is, a $4$-splitting family in which every set has size $\lfloor k/2 \rfloor$). Later, \cite{stinson-mt} stated that this lower bound holds for arbitrary $4$-splitting families, citing the aforementioned article as proof. In the following result we repair this minor error and prove that the claimed lower bound does indeed hold for arbitrary $4$-splitting families.

\begin{theorem}
	Let $k\geq 6$ and suppose $\mathcal A$ is a $4$-splitting family on $k$. Then $|\mathcal A|\geq\log_2{k}$.
\end{theorem}

In the proof and later in the paper, we will use the following terminology. Let $\mathcal A=\{A_1,A_2,\ldots, A_n\}$ be a family of subsets of $[k]$. For any $I\subset \{1,2,\ldots,n\}$, the \emph{Venn region} corresponding to $I$ is the set
\[\bigcap\{A_i\mid i\in I\}\cap\bigcap\{A_i^c\mid i\notin I\}\;.
\]
The \emph{multiplicity} of the Venn region corresponding to $I$, or any of its elements, is $|I|$. We will say two Venn regions corresponding to $I,I'$ are \emph{adjacent} if $|I\triangle I'|=1$.

\begin{proof}
	Assume towards a contradiction that $n<\log_2{k}$ so that $k\geq 2^n + 1$.
	
	We consider the subfamily $\mathcal A_0=\{A_1,A_2,\ldots A_{n-1}\}$. Then $\mathcal A_0$ has $2^{n-1}$ Venn regions, and since $k>2^n=2\cdot 2^{n-1}$ we conclude that some Venn region $R$ of $\mathcal A_0$ contains at least three points $x,y,z$.
  
  We can replace $A_n$ with its complement if necessary, so we can assume without loss of generality that $A_n$ contains at least two of $x,y,z$, say $x$ and $y$. If $w\in[k]$ is any other point, consider the set $\{x,y,z,w\}$. $\mathcal A_0$ does not split it since $x,y,z \in R$. Therefore, since $\mathcal A$ is a $4$-splitting family, $A_n$ must split $\{x,y,z,w\}$, implying $z,w \notin A_n$. Since $w$ was chosen arbitrarily, it follows that $A_n = \{x,y\}$. The $k-3\geq 2^n-2$ points other than $x,y,z$ are distributed among the $2^{n-1}$ Venn regions of $\mathcal A_0$; we consider several cases on how these points are distributed.
	
	\underline{Case 1:} There is a point in $R$ besides $x,y,z$.
	
	Let $w$ be such a point. We know $w\not \in A_n$. Since $k\geq 6$, we can find a fifth point $w'$, and again $w'\notin A_n$. Then $\{x,z,w,w'\}$ is not split by $A_n$ since $A_n$ only contains $x$. Further, $\{x,z,w,w'\}$ is not split by any set from $\mathcal A_0$ since $x,z,w$  are in the same Venn region of $\mathcal A_0$. Thus, $\{x,z,w,w'\}$ is not split by any set from $\mathcal A$, a contradiction.
	
 	\underline{Case 2:} There are no points in $R$ besides $x,y,z$ and some Venn region of $\mathcal A_0$ other than $R$ has at least three points.
	 
  Let $R'$ be such a region and let $w,w',w''\in R'$. Then $\{z,w,w',w''\}$ is not split by any set from $\mathcal A_0$. Moreover $A_n$ contains none of these four points, so $\{z,w,w',w''\}$ is not split by any set from $\mathcal A$, a contradiction.

 	\underline{Case 3:} There are no points in $R$ besides $x,y,z$ and no Venn region of $\mathcal A_0$ other than $R$ has at least three points.
	
	The $2^n-2$ points other than $x,y,z$ are distributed among the $2^{n-1}-1$ Venn regions of $\mathcal A_0$ other than $R$. In this case, since $2^n-2=2\cdot\left( 2^{n-1}-1\right)$, every such region must have exactly two points. Choose distinct Venn regions $R'$ and $R''$ which are adjacent to $R$. (This is possible because the assumption $k\geq 6$ implies $|\mathcal A|\geq3$, so $|\mathcal A_0|\geq 2$.) Letting $w'\in R'$ and $w''\in R''$, we have that $\{x,z,w',w''\}$ is not split by any set from $\mathcal A_0$. Moreover $\{x,z,w',w''\}$ is not split by $A_n=\{x,y\}$. Once again $\{x,z,w',w''\}$ is not split by any set from $\mathcal A$, a contradiction.
	
	This completes the proof.
\end{proof}

We now turn to $\mathord{\leq}4$-splitting families.

\begin{theorem}
  Let $k\geq5$ and suppose $\mathcal A$ is a $\mathord{\leq}4$-splitting family on $k$. Then $|\mathcal A|\geq\log_2k+3-\log_2 5$.
\end{theorem}

\begin{proof}
  Suppose $\mathcal A=\{A_1,\ldots,A_n\}$ and this time let $\mathcal A_0=\{A_1,\ldots,A_{n-2}\}$ and $\mathcal A_1 = \{A_{n-1}, A_n\}$.
  
  We first claim that no Venn region of $\mathcal A_0$ contains four points. Indeed if there were such $x,y,z,w$, then since $\mathcal A$ is $2$-splitting we would have that each of $x,y,z,w$ lie in distinct Venn regions of $\mathcal A_1$. Since we assumed $k\geq 5$, we can let $u$ be any other point. Without loss of generality, $u$ is in the same Venn region of $\mathcal A_1$ as $x$. But then the $4$-element set $\{x,y,z,u\}$ is not split, a contradiction.
  
  We next claim there cannot exist Venn regions $R_1,R_2,R_3$ of $\mathcal A_0$, each with three points, such that $R_1\sim R_2\sim R_3$. Indeed, otherwise we would be able to find $x_i$ in each $ R_i$ such that all three $x_i$ are in the same Venn region of $\mathcal A_1$. (This holds because three $3$-element subsets of a $4$-element set must have a point in common.) Suppose without loss of generality that $x_1,x_2,x_3\in A_{n-1}\cap A_n$. Let $y_2$ be another element of $R_2$ which is in a Venn region of $\mathcal A_1$ adjacent to that of $x_2$. Then $\{x_1,x_2,y_2,x_3\}$ is not split by any set of $\mathcal A$. Indeed $A_{n-1}$ and $A_n$ contain three of these points, and $A_1,\ldots,A_{n-2}$ each agree on $x_2,y_2$ and at least one of $x_1,x_3$.

  Now let $S$ be the set of Venn regions of $\mathcal A_0$ with three points. We claim that $|S|\leq2^{n-3}$. For this, observe that the adjacency graph of the $2^{n-2}$ Venn regions of $\mathcal A_0$ is isomorphic to a hypercube and thus may be decomposed into $2^{n-4}$ disjoint squares. If $|S|>2^{n-3} = 2 \cdot 2^{n-4}$ then one of these squares must contain three elements of $S$. It follows that we can find three elements $R_1,R_2,R_3\in S$ such that $R_1\sim R_2\sim R_3$. This contradicts the previous claim.
  
  We now have that the number of elements $k$ satisfies:
  \[k\leq 3|S|+2|S^c| = 2(|S|+|S^c|) + |S|
    \leq 2\cdot 2^{n-2} + 2^{n-3}
    = 5\cdot 2^{n-3}\text{.}
  \]
  This implies the desired bound.
\end{proof}

It is remarkable that the lower bounds in the previous two theorems are so close to one another. The distance between them is more visible if one considers the inverse functions, that is, the corresponding upper bounds on the $k$ for which there exists a $4$- or $\mathord{\leq}4$-splitting family on $k$ of a given size. Of course, these bounds are unlikely to be tight.

\section{Splittable families}

Recall from the introduction that a family $\mathcal B = \{B_1, \ldots, B_n\}$ of subsets of $[k]$ is called \emph{splittable} if there exists a single set $A \subset [k]$ such that $A$ splits $B_i$ for all $i\leq n$. While some splittable families have many distinct splitters, others are ``just barely splittable'' in the sense that they have very few splitters. In this section we investigate the question: which $n$-set splittable family on $k$ has the fewest number of splitters, and what is this number? Solutions to this question can be used to find bounds on the least size of an ``$n$-splitting family'' (see \cite{reu14} for details). Our solution is somewhat involved even when $n=1,2$, and we provide computational results when $n=3$.

We first consider the case when $n=1$, that is, $\mathcal B=\{B\}$ and $B\subset[k]$. If $|B|$ is even then the number of splitters of $\mathcal B$ is
\[2^{k-|B|}\cdot \binom{|B|}{|B|/2}\text{,}
\]
and if $|B|$ is odd then the number of splitters of $\mathcal B$ is
\[2^{k-|B|}\cdot \left(\binom{|B|}{|B|/2 + 1}+\binom{|B|}{|B|/2 - 1}\right) = 2^{k-|B|+1}\binom{|B|}{(|B|-1)/2}\text{.}
\]
The following result is unsurprising, and we record the proof.

\begin{proposition}
  The number of splitters of $\mathcal B=\{B\}$ on $[k]$ is minimized when $|B|=k$ if $k$ is even, and when $|B| = k-1$ when $k$ is odd. The minimum number of splitters is asymptotic to $2^k/\sqrt{k}$.
\end{proposition}

\begin{proof}
	We first claim that if $|B|=k-i$ is odd (and nonempty), then removing one point from $B$ decreases the number of splitters. Indeed the number of splitters of $B$ is
  \[2^{i+1}\binom{k-i}{(k-i-1)/2}\text{,}
  \]
	and the number of splitters of $B$ with a point removed is
  \[2^{i+1} \binom{k-i-1}{(k-i-1)/2}\text{.}
  \]
	It is not difficult to calculate that the latter expression is smaller than the first whenever $k-i \geq 3$. (Note that when $k-i=1$, the two expressions are equal.)
  
  In particular the number of splitters is minimized when $|B|$ is even. We next claim that if $|B|=k-i$ is even (and nonempty), then removing two points from $B$ increases the number of splitters, that is:
  \[2^i \binom{k-i}{(k-i)/2} < 2^{i+2} \binom{k-i-2}{(k-i-2)/2} \text{.}
  \]
  This comes from applying Pascal's identity:
	\begin{align*}
		2^i\binom{k-i}{\frac{k-i}2} &= 2^i\left(\binom{k-i-2}{\frac{k-i-2}2-1} + 2\binom{k-i-2}{\frac{k-i-2}2} + \binom{k-i-2}{\frac{k-i-2}2+1}\right) \\
		&< 2^{i+2}\binom{k-i-2}{\frac{k-i-2}2} \text{.}
	\end{align*}
	The desired result follows.
  
  For the second statement, we now see that when $k$ is even the minimum number of splitters of a $1$-set family on $k$ is exactly $\binom{k}{k/2}$. It is a standard application of Stirling's approximations to conclude that this is asymptotic to $2^k/\sqrt{k}$.
\end{proof}

We now turn to the case when $n=2$. We begin by fixing some notation. Given a family $\mathcal B = \{B_1, B_2\}$, we say the \emph{arrangement} of $\mathcal B$ is the quadruple $(a_1,b,a_2,d)$, where $a_1=|B_1\backslash B_2|$, $b=|B_1\cap B_2|$, $a_2=|B_2\backslash B_1|,$ and $d=|B_1^c\cap B_2^c|$ (see Figure~\ref{fig:two-set}). The number of splitters of $\mathcal B$ is determined by its arrangement, and we will often use the family and the arrangement interchangeably. We let $\splitters(\mathcal B)$ and $\splitters(a_1,b,a_2,d)$ both denote the number of splitters of $\mathcal B$.

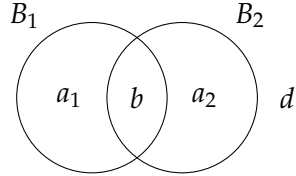
\begin{figure}[H]
	\center
	\begin{tikzpicture}
		\draw (0:.6) circle (1);
		\draw (180:.6) circle (1);
		\node at (-.9,0) {$a_1$};
		\node at (0,0) {$b$};
		\node at (.9,0) {$a_2$};
		\node at (2,0) {$d$};
    \node at (-1.5,1.1) {$B_1$};
    \node at (1.5,1.1) {$B_2$};
	\end{tikzpicture}
	\caption{A family $\mathcal B=\{B_1,B_2\}$ with arrangement $(a_1,b,a_2,d)$.\label{fig:two-set}}
\end{figure}

We now proceed to calculate a formula for $\splitters(a_1,b,a_2,d)$. Initially suppose that $|B_1|=a_1+b$ and $|B_2|=a_2+b$ are both even. Then if $A$ contains $i$ elements from $B_1 \cap B_2$, then $A$ is a splitter if and only if $A$ contains exactly $\frac{a_1 + b}2 - i$ elements from $B_1 \backslash B_2$ and $\frac{a_2 + b}2 - i$ elements from $B_2 \backslash B_1$. Thus the total number of splitters of $\mathcal B$ is given by:
\[\splitters(a_1, b, a_2, d) = 2^d\sum_{i=0}^b \binom{a_1}{\frac{a_1+b}{2}-i}
  \binom{b}{i}\binom{a_2}{\frac{a_2+b}{2}-i}\text{.}
\]
(Here if $z\in\mathbb{N}$ and $x\neq0,\ldots,z$ we define $\binom{z}{x}=0$.)

Note that if $|B_1|=a_1+b$ is odd and $|B_2|=a_2+b$ is even, then
\begin{align*}
	\splitters(a_1, b, a_2, d)
  &= 2^d \sum_{\epsilon = 0}^1 \sum_{i = 0}^b
         \binom{a_1}{\floor{\frac{a_1 + b}{2}} + \epsilon - i}
         \binom{b}{i}\binom{a_2}{\frac{a_2 + b}{2} - i} \\
	&= 2^d\sum_{i = 0}^b
        \binom{a_1 + 1}{\frac{(a_1 + 1) + b}{2} - i}
        \binom{b}{i}\binom{a_2}{\frac{a_2 + b}{2} - i} \\
	&= \splitters(a_1 + 1, b, a_2, d)\;.
\end{align*}
Similarly if both $|B_1| = a_1 + b$ and $|B_2| = a_2 + b$ are odd, we have
\[\splitters(a_1, b, a_2, d) = \splitters(a_1+1,b,a_2+1,d)\;.
\]

We are now ready to give the characterization of $2$-set families with the minimum number of splitters.

\begin{theorem}
  \label{min2}
  Let $k=a_1+b+a_2+d$ be fixed. Then $\splitters(a_1,b,a_2,d)$ is minimized by the following arrangements:
  \begin{itemize}
    \item if $k\equiv0\pmod{3}$, minimum at $(m,m,m,0)$, where $m=\frac{k}3$;
    \item if $k\equiv1\pmod{3}$, minimum at $(m-2,m,m,0)$, where $m=\frac{k+2}3$, and;
    \item if $k\equiv2\pmod{3}$, minimum at $(m,m,m+2,0)$, where $m=\frac{k-2}3$.
  \end{itemize}
  Moreover, the minimum number is asymptotic to $2^k/k$.
\end{theorem}

The proof consists of a series of ``point-moving'' lemmas. Each one shows that given an arbitrary arrangement, we can find a way to nudge it towards the proposed minimum arrangement without increasing the number of splitters.

\begin{lemma}
	\label{swap lemma}
	If $a_1+b$ and $a_2+b$ are even and $d\geq0$, then for any permutation $(a_1', b', a_2')$ of $(a_1, b, a_2)$, we have $\splitters(a_1, b, a_2, d) = \splitters(a_1', b', a_2', d)$.
\end{lemma}

\begin{proof}
  It suffices to consider the case when $d=0$, since otherwise both sides of the inequality simply have an extra factor of $2^d$. Further it suffices to show that $\splitters(a_1,b,a_2,0)=\splitters(b,a_1,a_2,0)$, since by symmetry $\splitters(a_1,b,a_2,0)=\splitters(a_2,b,a_1,0)$, and all other permutations are generated by these two.
  
  To show that $\splitters(a_1,b,a_2,0)=\splitters(b,a_1,a_2,0)$, fix a family $\mathcal B = \{B_1, B_2\}$ with arrangement $(a_1,b,a_2,0)$ and let $\mathcal B'=\{B_1, B'_2\}$, where $B'_2 = B_1 \triangle B_2$, so that $\mathcal B'$ has arrangement $(b,a_1,a_2,0)$. Next define a mapping of sets $A\mapsto A'$ by $A' = (A \cap B_1) \cup ((B_2 \backslash B_1) \backslash A)$. We claim that $A\mapsto A'$ is an injection from the splitters of $\mathcal B$ to the splitters of $\mathcal B'$.
	
	We first argue that if $A$ splits $\mathcal B$ then $A'$ splits $\mathcal B'$. Since $A$ splits $B_1$ and $A' \cap B_1 = A \cap B_1$, $A'$ splits $B_1$ as well. By the assumption that $a \equiv b \equiv c\pmod{2}$, all of $B_1,B_2,B_2'$ are even-sized. Now if $A$ contains $i$ elements from $B_1 \cap B_2$, then $A$ contains $\frac{a + b}{2} - i$ elements from $B_1 \setminus B_2$ and $\frac{b + c}{2} - i$ elements from $B_2 \backslash B_1$. Thus $A'$ contains $i$ elements from $B_1 \setminus B'_2$ and $\frac{a + b}{2} - i$ elements from $B_1 \cap B'_2$ and $c - (\frac{b + c}{2} - i) = \frac{c - b}{2} + i$ elements from $B'_2 \setminus B_1$. Thus $A'$ contains exactly $\frac{a + c}{2}$ elements from $B'_2$, so $A'$ splits $\mathcal B'$.
  
  Finally note that since $d=0$, we can recover $A$ from $A'$ by the formula $A=(A'\cap B_1)\cup((B_2\setminus B_1)\setminus A')$. This shows that $A\mapsto A'$ is injective, so $\splitters(\mathcal B)\leq\splitters(\mathcal B')$. Moreover the symmetry of the inverse mapping shows that $\splitters(\mathcal B')\leq\splitters(\mathcal B)$, concluding the proof.
  %
\end{proof}

Recall that the definition of splitting allows rounding up or down in the case of odd-sized sets. Thus if a family has an odd-sized set, one would expect it to have more splitters than a similar family with even-sized sets only. The following result confirms this intuition.

\begin{lemma}
  \label{no_odd}
	Suppose that $a_1+b$ and $a_2+b$ are even, and let $\epsilon_1,\epsilon_2\in \{-1,0,1\}$. Then
  \[\splitters(a_1, b, a_2, 0)\leq\splitters(a_1+\epsilon_1, b-\epsilon_1-\epsilon_2, a_2+\epsilon_2,0)\text{.}
  \]
\end{lemma}

\begin{proof}
	Let $\mathcal B=\{B_1,B_2\}$ be a family with arrangement $(a_1,b,a_2,0)$, and let $\mathcal B'=\{B_1',B_2'\}$ be the family with arrangement $(a_1+\epsilon_1, b-\epsilon_1-\epsilon_2, a_2+\epsilon_2,0)$ that is derived from $\mathcal B$ according to the $\epsilon_i$ in the natural way. (For example, if $\epsilon_1 = 1$ and $\epsilon_2 = 0$, we can construct $B'$ from $B$ by selecting some point $x \in B_1\cap B_2$ and placing it in $B_1'\setminus B_2'$.) It is easy to see that if $A$ splits $\mathcal B$, then $A$ splits $\mathcal B'$ too, completing the proof.
\end{proof}

The next result allows one to move points from outside the sets $B_1,B_2$ into the sets without increasing the number of splitters.

\begin{lemma}
  \label{ext_zero}
  For fixed $k$, for any arrangement $(a_1,b,a_2,d)$ on $k$, there exists an arrangement $(a_1',b',a_2',0)$ on $k$ such that
  \[\splitters(a_1',b',a_2',0)\leq\splitters(a_1,b,a_2,d)\text{.}
  \]
\end{lemma}

\begin{proof}
	By Lemma \ref{no_odd}, we need only consider the case when $a_1+b$ and $a_2+b$ are both even. By Lemma~\ref{swap lemma} we can suppose that $b>0$ (otherwise $d=k$ and this is clearly not minimal). We now use the result \cite[Theorem~3.11]{reu14}, which states that if $b,d>0$ then we have:
  \[\splitters(a_1+1,b-1,a_2+1,d-1)\leq\splitters(a_1,b,a_2,d)\text{.}
  \]
  Note that the arrangement on the left-hand side again satisfies the hypotheses of Lemma~\ref{swap lemma}. Thus we can repeat the process inductively to obtain the desired conclusion.
\end{proof}

The next result is the key, as it allows one to move $2$ points from one of the sets to the other without increasing the number of splitters, provided it would make the regions more balanced in size.

\begin{lemma}
  \label{2 point lemma}
  Assume that $a_1+b$ and $a_2+b$ are even. If $2\leq a_1\leq a_2$ then
  \[\splitters(a_1,b,a_2,0)\leq\splitters(a_1-2,b,a_2+2,0) \text{.}
  \]
\end{lemma}

\begin{proof}
  Let $\mathcal B=\{B_1,B_2\}$ be a family with arrangement $(a_1,b,a_2,0)$ and $\mathcal B'=\{B_1',B_2'\}$ be a family with arrangement $(a_1-2,b,a_2+2,0)$. We can again assume $\mathcal B'$ is constructed from $\mathcal B$ in the natural way, that is, fix $x,y\in B_1\backslash B_2$ and let $B_1' = B_1 \setminus \{x,y\}$ and $B_2' = B_2 \cup \{x,y\}$. We wish to show that there exists an injection from the splitters of $\mathcal B$ to the splitters of $\mathcal B'$. Note immediately that if $A$ is a splitter of $\mathcal B$ that contains either $x$ or $y$, but not both, then $A$ also splits $\mathcal B'$. Thus it remains to show that there are fewer splitters of $\mathcal B$ that contain (omit) both $x,y$ than splitters of $\mathcal B'$ that contain (omit) both $x,y$.

  We address the ``omit'' case with the ``contain'' case being similar. Let $S$ be the number of splitters of $\mathcal B$ that omit $x,y$, and $S'$ be the number of splitters of $\mathcal B'$ that omit $x,y$. We wish to show that $S\leq S'$. To proceed, let us first assume that $b$ is even. Let $t_i=\frac{a_i+b}{2}$, the \emph{target} number of elements of $B_i$ for a splitter $A$ of $\mathcal B$. We calculate:
  \begin{align*}
    S'-S
      &=\sum_{i=0}^b
      \binom{a_1-2}{t_1-i-1}\binom{b}{i}\binom{a_2}{t_2-i+1}
      -\binom{a_1-2}{t_1-i}\binom{b}{i}\binom{a_2}{t_2-i}\\
      &=\sum_{i=0}^{b/2}
      \binom{a_1-2}{t_1-i-1} \binom{b}{i} \binom{a_2}{t_2-i+1}
      -\binom{a_1-2}{t_1-i} \binom{b}{i} \binom{a_2}{t_2 - i} \\
      &\qquad+\binom{a_1-2}{t_1-(b+1-i)} \binom{b}{b+1-i} \binom{a_2}{t_2-(b+1-i)+1}\\
      &\qquad-\binom{a_1-2}{t_1-(b+1-i)} \binom{b}{b+1-i} \binom{a_2}{t_2-(b+1-i)} \\
      &=\sum_{i = 0}^{b/2}
      \binom{b}{i}
      \left(\binom{a_1-2}{t_1-i-1} \binom{a_2}{t_2-i+1}
      -\binom{a_1-2}{t_1-i} \binom{a_2}{t_2-i}\right) \\
      &\qquad+\binom{b}{i - 1}
      \left(\binom{a_1-2}{t_1-i} \binom{a_2}{t_2-i}
      -\binom{a_1-2}{t_1-i-1} \binom{a_2}{t_2-i+1}\right) \\
      &=\sum_{i=0}^{b/2}
      \left(\binom{b}{i}-\binom{b}{i-1}\right)
      \left(\binom{a_1-2}{t_1-i-1}\binom{a_2}{t_2-i+1}
      -\binom{a_1-2}{t_1-i}\binom{a_2}{t_2-i}\right)
  \end{align*}
  The second line above was the key step, wherein we paired term $i$ with term $j=b+1-i$. In the third equality we reflect five binomial coefficients and observe that $a_i-t_i=t_i-b$.

  Clearly, $\binom{b}{i}-\binom{b}{i-1}$ is positive for each $i\leq b/2$. Thus it is enough to show that for all $i\leq b/2$ we have
  \begin{equation}
    \label{eq:term}
    \binom{a_1-2}{t_1-i-1}\binom{a_2}{t_2-i+1}-\binom{a_1-2}{t_1-i}\binom{a_2}{t_2-i}
    \geq0\text{.}
  \end{equation}
  Note that if the second term of inequality~\eqref{eq:term} is zero, we are done. On the other hand if the second term of inequality~\eqref{eq:term} is nonzero, then we claim that the first term is nonzero as well. To see this, the indices where the first term is nonzero are $[t_1-a_1+1,t_1-1]\cap[t_2-a_2+1,t_2+1]=[t_1-a_1+1,t_1-1]$, and the indices where the second term is nonzero are $[t_1-a_1+2,t_1]\cap[t_2-a_2,t_2]=[t_1-a_1+2,t_1]$. (We are using here that $a_1\leq a_2$, $t_1\leq t_2$, and $t_2-a_2\leq t_1-a_1$.) The only index in the latter set but not the former is $i=t_1$, but the hypothesis $2\leq a_1$ implies $b/2<t_1$, and so this index is not in the sum.
    
  Now for indices $i$ such that the terms of inequality~\eqref{eq:term} are nonzero, we have:
  \begin{equation}
    \label{eq:frac}
    \frac{\binom{a_1-2}{t_1-i-1}\binom{a_2}{t_2-i+1}}
         {\binom{a_1-2}{t_1-i}\binom{a_2}{t_2-i}}
      =\frac{(t_1-i)(a_2-t_2+i)}{(a_1-t_1+i-1)(t_2-i+1)}
      =\frac{(t_1-i)(t_2-b+i)}{(t_1-b+i-1)(t_2-i+1)}.
  \end{equation}
  It suffices to show that the last quantity in equation~\eqref{eq:frac} is $>1$. Recall that the Mediant Inequality states that if $A,B,C,D>0$ and $\frac{A}{B}>\frac{C}{D}$ then $\frac{A}{B}>\frac{A+C}{B+D}>\frac{C}{D}$. Since $t_1\leq t_2$ we can write:
  \[\frac{t_2 - b + i}{t_1 - b + i - 1} > 1 = \frac{b - 2i + 1}{b - 2i + 1}\text{.}
  \]
  Since we are assuming the numerator and denominator of the left-hand side are positive, and since $b-2i+1$ is positive as well, the Mediant Identity implies that
  \[\frac{t_2 - b + i}{t_1 - b + i - 1}
  >\frac{t_2 - i + 1}{t_1 - i}
  >\frac{b - 2i + 1}{b - 2i + 1}\text{.}
  \]
  This implies that the last quantity in Equation~\eqref{eq:frac} is $>1$, as desired.

  If $b$ is odd then the calculation of $S'-S$ is similar. We again pair the $i$ and $b+1-i$ terms. Of course the $i=(b+1)/2$ term has nothing to pair with, but fortunately it cancels out completely. To see this note that the left and right terms are reflections of one another, since
  \begin{align*}
    (t_1-(b+1)/2-1)+(t_1-(b+1)/2)&=a_1-2,\text{ and}\\
    (t_2-(b+1)/2+1)+(t_2-(b+1)/2)&=a_2\text{.}
  \end{align*}
  Thus in the second line we can take the sum from $i=1$ up to $i=(b-1)/2$. The rest of the proof is the same as before.
\end{proof}

The combination of Lemmas~\ref{swap lemma} and~\ref{2 point lemma} means that (when $d=0$) moving two points from any region to another region with fewer points does not increase the number of splitters. We are now ready to prove our main result.

\begin{proof}[Proof of Theorem~\ref{min2}]
  We will give the proof when $k = 3m$, the remaining cases are similar. By the lemmas together, it suffices to show that for any even $x,y\geq0$ we have:
  \[\splitters(m,m,m,0)\leq\splitters(m+x, m-x+y, m-y,0)\;.
  \]
  We show this in two steps, with each step consisting of several applications of Lemma~\ref{2 point lemma} and~\ref{swap lemma}. Assuming $x\leq y$, first use the two lemmas to achieve:
  \[\splitters(m,m,m,0) \leq \splitters(m+x,m,m-x,0)\text{.}
  \]
  We then use the two lemmas again to achieve:
  \[\splitters(m+x,m,m-x,0) \leq \splitters(m+x, m-x+y, m-y,0)\text{.}
  \]
  We refer to Figure~\ref{fig:lattice} for a visualization of these two inequalities. The steps are similar if $y\leq x$. This concludes the proof that $\splitters(m,m,m,0)$ is minimal.
  
  For the last statement, observe that $\splitters(k/3,k/3,k/3,0)$ is exactly equal to $\sum_i\binom{k/3}{i}^3$. These are \emph{Franel numbers}, and it is not hard to see this expression is asymptotic to $2^k/k$. (See for instance \cite{franel}.)
\end{proof}

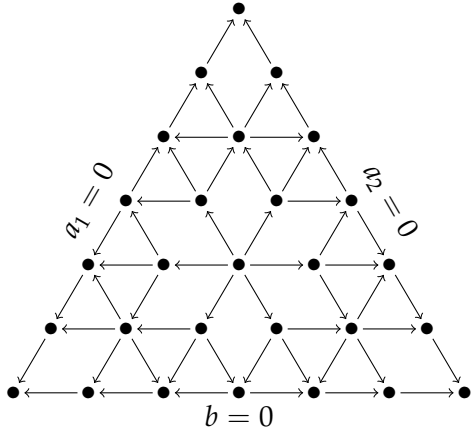
\begin{figure}[ht]
\centering
\begin{tikzpicture}[inner sep=1]
	\node at (-3, -1.7) (060) {$\bullet$};
	\node at (-2, -1.7) (051) {$\bullet$};
	\node at (-1, -1.7) (042) {$\bullet$};
	\node at (0, -1.7) (033) {$\bullet$};
	\node at (1, -1.7) (024) {$\bullet$};
	\node at (2, -1.7) (015) {$\bullet$};
	\node at (3, -1.7) (006) {$\bullet$};
	\node at (-2.5, -0.85) (150) {$\bullet$};
	\node at (-1.5, -0.85) (141) {$\bullet$};
	\node at (-0.5, -0.85) (132) {$\bullet$};
	\node at (0.5, -0.85) (123) {$\bullet$};
	\node at (1.5, -0.85) (114) {$\bullet$};
	\node at (2.5, -0.85) (105) {$\bullet$};
	\node at (-2, 0) (240) {$\bullet$};
	\node at (-1, 0) (231) {$\bullet$};
	\node at (0, 0) (222) {$\bullet$};
	\node at (1, 0) (213) {$\bullet$};
	\node at (2, 0) (204) {$\bullet$};
	\node at (-1.5, 0.85) (330) {$\bullet$};
	\node at (-0.5, 0.85) (321) {$\bullet$};
	\node at (0.5, 0.85) (312) {$\bullet$};
	\node at (1.5, 0.85) (303) {$\bullet$};
	\node at (-1, 1.7) (420) {$\bullet$};
	\node at (0, 1.7) (411) {$\bullet$};
	\node at (1, 1.7) (402) {$\bullet$};
	\node at (-0.5, 2.55) (510) {$\bullet$};
	\node at (0.5, 2.55) (501) {$\bullet$};
	\node at (0, 3.4) (600) {$\bullet$};
	\node at (0, -2) {$b = 0$};
	\node[rotate=60] at (-2, .85) {$a_1 = 0$};
	\node[rotate=-60] at (2, 0.85) {$a_2 = 0$};
	\draw[->](051)--(060);\draw[->](042)--(051);\draw[->](033)--(042);\draw[->](033)--(024);
	\draw[->](024)--(015);\draw[->](015)--(006);\draw[->](141)--(150);\draw[->](132)--(141);
	\draw[->](123)--(114);\draw[->](114)--(105);\draw[->](231)--(240);\draw[->](222)--(231);
	\draw[->](222)--(213);\draw[->](213)--(204);\draw[->](321)--(330);\draw[->](312)--(303);
	\draw[->](411)--(420);\draw[->](411)--(402);\draw[->](510)--(600);\draw[->](420)--(510);
	\draw[->](330)--(420);\draw[->](330)--(240);\draw[->](240)--(150);\draw[->](150)--(060);
	\draw[->](411)--(501);\draw[->](321)--(411);\draw[->](231)--(141);\draw[->](141)--(051);
	\draw[->](312)--(402);\draw[->](222)--(312);\draw[->](222)--(132);\draw[->](132)--(042);
	\draw[->](213)--(303);\draw[->](123)--(033);\draw[->](114)--(204);\draw[->](114)--(024);
	\draw[->](105)--(006);\draw[->](204)--(105);\draw[->](303)--(204);\draw[->](303)--(402);
	\draw[->](402)--(501);\draw[->](501)--(600);\draw[->](114)--(015);\draw[->](213)--(114);
	\draw[->](312)--(411);\draw[->](411)--(510);\draw[->](123)--(024);\draw[->](222)--(123);
	\draw[->](222)--(321);\draw[->](321)--(420);\draw[->](132)--(033);\draw[->](231)--(330);
	\draw[->](141)--(042);\draw[->](141)--(240);
\end{tikzpicture}
\caption{The lattice of arrangements $(a_1,b,a_2,0)$ with $k=12$ and $a_1+b,a_2+b$ even. An arrow $\alpha\to\beta$ denotes $\alpha$ has fewer or equal splitters than $\beta$. The proof of Theorem~\ref{min2} uses the fact that one can follow the arrows from the centroid to any other point using two straight-line steps.\label{fig:lattice}}
\end{figure}

While we kept it out of the statement for cleanliness, we record here that the proof of Theorem~\ref{min2} may be used to classify all minimum arrangements. First observe that Lemma~\ref{swap lemma} implies there are two additional minimum arrangements in the cases when $k\equiv1,2\pmod{3}$. That these are exhaustive may be seen from the strict inequality obtained in the end of the proof of Lemma~\ref{2 point lemma}.

Before leaving the case when there are $n=2$ sets, we mention an approximate formula for the value of $\splitters(a_1,b,a_2,d)$. Recall that the de~Moivre--Laplace Theorem \cite[Section~7.3]{feller} gives the approximation
\[\binom{n}{k} \approx 2^n \sqrt{\frac{2}{\pi n}} e^{-\frac{2}{n} \left(k-\frac{n}{2}\right)^2}.
\]
Thus when $a_1 \equiv b \equiv a_2\pmod{2}$ we can approximate the number of splitters by
\begin{align*}
	\splitters(a_1, b, a_2,d) &= 2^{d}\sum_{i = 0}^b \binom{b}{i}\binom{a_1}{t_1 - i}\binom{a_2}{t_2 - i} \\
	&\approx 2^{a_1 + b + a_2+d} \sqrt{\frac{8}{\pi^3 a_1 b a_2}}\int_{-\infty}^\infty e^{-\frac{2}{b}\left(x - \frac{b}{2}\right)^2 - \frac{2}{a_1}\left(\frac{b}{2} - x\right)^2 - \frac{2}{a_2}\left(\frac{b}{2} - x\right)^2}dx \\
	&= \frac{2^{k + 1}}{\pi \sqrt{a_1 a_2 + a_1 b + a_2 b}}.
\end{align*}
It is not difficult to calculate that the latter expression has its minimum when $a_1=b=a_2=k/3$ and $d=0$. However without very tight control over the error in the approximation, it would not be possible to use this information to replace the proof of the previous theorem.

We close this section by briefly considering the case when there are $n=3$ sets. The formula for the number of splitters when $n=3$ is considerably more complex than the $n=2$ case. Using an exhaustive search elaborated in Appendix~\ref{appendix:3-set}, we found that if $N_k$ denotes the minimum number of splitters of a splittable $3$-set family, then $N_k$ appears to obey $N_6=4$ and
\[\frac{N_{k+1}}{N_k} = 
  \left\{
  	\begin{array}{l l}
  		2 - \frac{1}{\floor{k/6} + 1}, & k \text{ even} \\
  		2, & k \text{ odd}
  	\end{array}
  \right.
\]
It is not difficult to show that the recurrence above implies that $N_k$ is asymptotic to $2^k/k^{3/2}$. Together with the known asymptotics for the case $n=2$, this supports the following.

\begin{conjecture}
	The minimum number of splitters of a splittable $n$-set family is asymptotic to
  \[2^k/k^{n/2} \; .
  \]
\end{conjecture}
  
\section{The splitting game}

In this section we again look for arrangements that are in some sense the least splittable. This time, rather than considering the number of splitters an arrangement has, we instead consider whether the arrangement can be split under adversarial conditions. To do this, we adopt a game theoretic perspective on the notion of set splitting.

The \emph{splitting game} is played by players Split and Skew on a game board $(k,\mathcal B)$, where $k$ is a positive integer and $\mathcal B$ is a family of subsets of $[k]$. An instance $(k, \mathcal B, t)$ of the game consists of a game board together with $t \in \{\mathrm{Split},\mathrm{Skew}\}$ indicating which player goes first. The players alternately claim an element from $[k]$, without repetition, until all the elements have been claimed. Split wins the game if the set of elements Split claimed splits $\mathcal B$ (\emph{i.e.}, splits every $B\in\mathcal B$), and Skew wins otherwise.

The splitting game lies in the general class of games known as Maker--Breaker games. Introduced by Paul Erd\"os and John Selfridge, such games consist of two players choosing objects with Maker trying to occupy some winning arrangement and Breaker trying to prevent Maker's success \cite{beck-tic}. Since such games are finite and of perfect information, Zermelo's theorem implies that in each instance of the game, one of the two players must have a winning strategy.
  

It is clear that if Split has a winning strategy in $(k,\mathcal B, t)$, then $\mathcal B$ is splittable. However, the converse is not true. For example, consider a game on $k=3$ consisting of the sets $\mathcal B = \{\{1,2\}, \{2, 3\}\}$. Note that $\mathcal B$ is splittable with two splitters: $\{1, 3\}$ and $\{2\}$. Therefore, in the game $(3, \mathcal B, \text{Skew})$, if Skew first claims $2$, then Split can claim $1$ or $3$ but not both, resulting in a victory for Skew.

We begin with some general observations on the splitting game.

\begin{lemma}[Reduction lemma]
	\label{reduction lemma}
	Let $(k,\mathcal B),(k',\mathcal B')$ be game boards with $|\mathcal B| = |\mathcal B'|$, and suppose $|R| \equiv |R'|\pmod{2}$ for corresponding Venn regions $R,R'$ of $\mathcal B,\mathcal B'$. Then Player~$p$ has a winning strategy in $(k,\mathcal B,t)$ if and only if $p$ has a winning strategy in $(k',\mathcal B',t)$.
\end{lemma}

\begin{proof}
  It is enough to consider the case when $\mathcal B'$ is formed by adding a pair of points to some Venn region of $\mathcal B$. One can then use this case repeatedly, adding or removing pairs of points, to obtain the full statement. So assume that $k'=k+2$, and that $\mathcal B'$ is formed by adding the points $k+1,k+2$ to some subset of the sets in $\mathcal B$.
  
	Let $\sigma$ be a winning strategy for Player~$p$ in $(k,\mathcal B,t)$ and let Player~$q$ be the other player. We define a strategy for $p$ in $(k',\mathcal B',t)$ as follows. If $q$ claims one of $k+1$ or $k+2$, then $p$ claims the other immediately. Otherwise $p$ plays according to $\sigma$. If every element of $[k]$ is claimed before this happens, and it is $p$'s turn to move, then $p$ takes either one, and $q$ is forced to take the other.
	
	Let $A$ be the set of elements claimed by Player~$p$ after a run where $p$ played according to this strategy. Since $p$ followed $\sigma$, considering only turns involving $[k]$, $A$ splits or skews $\mathcal B$ as desired by $p$. Moreover since $A$ contains exactly one of $k+1$ and $k+2$, it follows that $A$ remains a splitter or a skewer, as the case may be, of $\mathcal B'$. In any case, $p$ wins $(k',\mathcal B',t)$. 
	
	Conversely, if Player~$p$ does not have a winning strategy in $(k,\mathcal B,t)$, then by Zermelo's Theorem, $q$ must have a winning strategy. By the forward implication, $q$ then has a winning strategy in $(k',\mathcal B',t)$. Thus $p$ does not have winning strategy in $(k',\mathcal B',t)$ if $p$ does not have a winning strategy in $(k,\mathcal B,t)$.
	%
\end{proof}

For the next result, if $\mathcal B$ is a family of subsets of $[k]$, we say that Split has a \emph{pairing strategy} for $\mathcal B$ if there exists a set of disjoint pairs $\mathcal P = \{P_1, \ldots, P_r\}$ of elements of $[k]$ such that every splitter of $\mathcal P$ splits $\mathcal B$. Similarly we say that Skew has a pairing strategy if every splitter of $\mathcal P$ does not split $\mathcal B$; that is, if $\mathcal P \cup \mathcal B$ is unsplittable.

\begin{lemma}[Pairing lemma]
  \label{pairing lemma}
  Let $\mathcal B$ be a family of subsets of $[k]$. If Player~$p$ has a pairing strategy for $\mathcal B$, then $p$ has a winning strategy for $(\ell,\mathcal B,t)$ for either value of $t$ and for each $\ell \geq k$.
\end{lemma}

\begin{proof}
	Suppose Split has a pairing strategy given by $\mathcal P=\{P_1, \ldots, P_r\}$, and consider the side game $(\ell,\mathcal P,t)$. By Lemma~\ref{reduction lemma}, Split has a winning strategy for $(\ell,\mathcal P,t)$ if and only if Split has a winning strategy in a game in which all sets are empty, which is clearly the case. Since any set that splits $\mathcal P$ also splits $\mathcal B$, this yields a winning strategy for $(\ell,\mathcal B,t)$.
	
	If Skew has a pairing strategy, we can use the same argument, but note that Skew will play the role of Split in the side game.
\end{proof}

In the rest of the section, we carry out case studies, analyzing which player wins in several special cases of the splitting game. We begin with a very trivial such case.

\begin{theorem}
	\label{even regions}
	Suppose that $(k,\mathcal B)$ has at most two odd-sized Venn regions of nonzero multiplicity. Then for either value of $t$, Split has a winning strategy in $(k,\mathcal B, t)$.
\end{theorem}

\begin{proof}
  By Lemma~\ref{reduction lemma}, the game board $(k,\mathcal B)$ is equivalent to a game board $(\ell, \mathcal C)$ in which 
$\left|\bigcup_{C \in \mathcal C} C\right| \leq 2.$
If $\left|\bigcup_{C \in \mathcal C} C\right| \leq 1$, then every subset of $[\ell]$ splits $\mathcal C$, so any run of the splitting game results in a Split win. If $\left|\bigcup_{C \in \mathcal C} C\right| = 2$, then Split has a pairing strategy given by $\left\{\bigcup_{C \in \mathcal C} C\right\}$.
\end{proof}

Next we proceed to study games where the number of sets $n=|\mathcal B|$ is small. By Lemma~\ref{reduction lemma}, one can reduce any game to one in which each Venn region has at most one element. Thus for fixed $n$ there are just finitely many $n$-set game boards after reduction, and for small $n$ we can exhaustively check who has a winning strategy in each case.

\begin{theorem}
	For all games with $|\mathcal B|\leq 3$, the winning player and corresponding winning strategy are known and identified in Appendix~\ref{appendix:two-three}.
\end{theorem}

This case study gives some insight into the question of which player wins more often. Let $P(n)$ be the proportion of $n$-set boards $(k, \mathcal B)$ (with at most one element in each Venn region) for which Split has a winning strategy for $(k, \mathcal B, \text{Split})$. From our data, we have $P(1)=1$, $P(2)=7/8$, and $P(3)=65/128$. Moreover, using the fact that games won by Skew are closed under adding sets, it is not difficult to show that $P(n)$ is a decreasing function of $n$, and in fact, we have the following result.

\begin{theorem}
	With $P(n)$ as described above, $P(n) \to 0$ as $n \to \infty$.
\end{theorem}

\begin{proof}
	Let $N \in \{0,1\}^{2^n}$ be chosen uniformly at random. Construct a game board $(k,\mathcal B)$ on $n$ sets from $N$ by identifying each of the $2^n$ Venn regions of $\mathcal B$ with a coordinate of $N$, and placing 0 or 1 points in each region, according to the corresponding coordinate of $N$. Then $P(n)$ is precisely the probability that Split has a winnning strategy for $(k, \mathcal B, \text{Split})$.
	
	Our strategy is to find subfamilies of $\mathcal B$ of size 2, with high probability, for which Skew has a winning strategy. Let $t = \lfloor n/2 \rfloor$, and let $\mathcal B_1, \ldots, \mathcal B_t$ be disjoint subfamilies of $\mathcal B$ on 2 sets. Let $S_i$ be the event that Split has a winning strategy for $(k,\mathcal B_i,\text{Split})$. Now consider a single Venn region $R$ of $\mathcal B_i$, and let $X$ be a random variable denoting $|R|$. Then $X$ is binomially distributed with $p=1/2$, so we have $\text{Pr}(X \text{ is even}) = \text{Pr}(X \text{ is odd}) = 1/2$. It follows from Lemma~\ref{reduction lemma} and Appendix~\ref{appendix:two-three} that $\text{Pr}(S_i) = P(2) = 7/8$.
	
	We claim that $S_1, \ldots, S_t$ are mutually independent. It suffices to show that for any nonempty subset $\{i_1, \ldots, i_s\}$ of $[t-1]$, we have $\text{Pr}(S_t \mid S_{i_1} \cap \dots \cap S_{i_s}) = \text{Pr}(S_t)$. Now recall that we are assigning 0 or 1 points to each Venn region of the total configuration. Imagine that we assign the regions contained in sets of $\mathcal B_1 \ldots \mathcal B_{t-1}$ first, thus determining the occurrence of  $S_{i_1} \cap \dots \cap S_{i_s}$. Note that each of the three interior Venn regions of $B_t$ has at least one Venn region of $\mathcal B$ unassigned. Therefore, the number of additional points assigned will follow a nontrivial binomial distribution, so the probability that an even (or odd) number of additional points will be assigned is $1/2$. Therefore, the initial assignment has no bearing on the probability of $S_t$, and we have mutual independence.
	
	Now, since $S_1 \cap \dots \cap S_t$ must occur in order for Split to have a winning strategy for $(k,\mathcal B,\text{Split})$, we have
	\[P(n) \leq \text{Pr}(S_1 \cap \dots \cap S_t) = (7/8)^t \to 0 \text{ as $n \to \infty$.}
  \]
  This completes the proof.
\end{proof}

It is also worth remarking that while we were able to give a pairing strategy for every $3$-set game, there is not always a pairing strategy. For an example with $n=6$, see the first statement of Theorem~\ref{thm:tictactoe} below. In this case we can show that Player~II has a winning strategy, so Lemma~\ref{pairing lemma} implies there cannot be a pairing strategy. For an example with $n=5$, we can also show Player~II has a winning strategy in the game on $\mathbb Z/5\mathbb Z$ whose board is described by $B_i=\{i,i+1,i+2\}$, $i = 1,\ldots, 5$. We do not know whether there is such an example using just four sets.


As there is significant literature on maker--breaker tic-tac-toe and its variants, we next study splitting tic-tac-toe. Here given $m,n$ we let $k=mn$ and think of $[k]$ as an $m\times n$ grid. We then let $\mathcal B$ consist of the rows and columns of the grid. Observe that $\mathcal B$ is splittable by a checkerboard pattern. For comparison, in $3\times3$ maker--breaker tic-tac-toe, Maker tries to make a line and Breaker tries to prevent Maker from making a line. In $3\times3$ splitting tic-tac-toe, Skew tries to make a line or cause Split to make a line.

Because of the special status of the $3\times3$ board, we first consider the ``classic'' variant of $3\times3$ splitting tic-tac-toe where $\mathcal B$ includes the diagonals as well as the rows and columns.

\begin{proposition}
	In ``classic'' $3\times3$ splitting tic-tac-toe, Skew has a winning strategy.
\end{proposition}

\begin{proof}
  If Skew is Player~I, begin by claiming the upper-left corner. Split is then forced to claim the middle, or else Skew can win traditional tic-tac-toe. Skew then claims the middle-top and middle-left, forcing Split to form a diagonal line.
  
  If Skew is Player~II, we consider the three cases depending on which element Split claims first. If Split claims the middle, then by claiming a corner, Skew can threaten to form three in a row no matter where Split goes next. Skew can take advantage of this by choosing one of the squares adjacent to its first move, forcing Split to take the final corner in that row or column. Thus, along this diagonal, either Split has already formed a full diagonal or there is one open square which Skew can force Split to claim as Split must make the last move.
  
  Next suppose Split claims the upper middle. Skew can then go in the middle and no matter what Split does next, Skew can occupy one of the two bottom corners. Thus, Skew threatens to form a diagonal so Split must prevent this by taking one of the two upper corners depending on where Skew went. Again, by avoiding the final spot in this row, Skew can force Split to take it, resulting in a full row for Split and a Skew win.
  
  Finally suppose Split claims the upper left. Skew again responds by selecting the middle. No matter where Split goes next, Skew can force Split to claim one of the two free squares adjacent to their first one. As before, Skew can simply avoid the third square in this row or column, eventually forcing Split to form three in a row.
  
  This covers all cases up to symmetry, so Skew has a winning strategy.
\end{proof}

We now return to the version of splitting tic-tac-toe without the diagonals.

\begin{theorem}
  \label{thm:tictactoe}
  Consider the $m\times n$ splitting tic-tac-toe game, with $m,n\neq1$.
  \begin{itemize}
    \item If $m=n=3$, Player~II has a winning strategy.
    \item Otherwise, Skew has a winning strategy.
  \end{itemize}
\end{theorem}

The proof consists of a number of cases and is carried out in Appendix~\ref{appendix:tictactoe}.

We conclude with the generalization of splitting tic-tac-toe to higher dimensions. Specifically, we play on a $d$-dimensional grid and let $\mathcal B$ consist of the maximal axis-parallel lines.

\begin{theorem}
  Consider the $n_1\times\cdots\times n_d$ game described above. Assume that $n_i\neq1$ for all $i$ (this would reduce to a lower dimensional game).
  \begin{itemize}
    \item If $d=2$ and $n_1=n_2=3$, then Player~II has a winning strategy.
    \item In all other cases, Skew has a winning strategy.
  \end{itemize}
\end{theorem}

\begin{proof}
  The first statement is part of the previous theorem.
  
  For the second statement, unless $n_i=3$ for all $i$, the previous theorem implies there is a two-dimensional sub-board where Skew has a winning strategy.
  
  If $n_i=3$ for all $i$, regard the board as a union of parallel $3\times3$ sub-boards. First suppose Skew is player I. After the first move, we can assume that Player~II moves on a different sub-board (otherwise we could cut the boards in a different way). Then Skew gets to make a second move in the same sub-board, giving them a one move lead. It is not difficult to see that Skew can now make a line in this sub-board. For example Skew's first four moves can form a $2\times2$ corner of the sub-board, forcing Split to play in other corners. After this Skew is threatening two lines and Split cannot block both.
  
  Next suppose that Skew is Player~II. After Split makes the first move, Skew can play on the same sub-board. If Split continues to play on the same sub-board, then Skew wins because Skew is Player~II in a $3\times3$ game. If Split plays outside the sub-board, Skew can copycat and also play outside the sub-board. Since there are an even number of elements in the complement of the sub-board, Skew can continue to use the winning strategy on the sub-board.
\end{proof}

\appendix
\section{$3$-set families with the minimum number of splitters}
\label{appendix:3-set}

Using a computer search, we found all types of splittable $3$-set families on all $k\leq 60$. The code and its output may be found in \cite{hao}.

For each of these $k$ there are several distinct families with the fewest number of splitters. However, the various solutions all look similar to one another (with some Venn regions permuted or single point moved somewhat). Every solution has either exactly two Venn regions with $1$ point or exactly one Venn region with $2$ points. The solution sets repeat in pattern every $6$ values of $k$.

In Figure~\ref{fig:3-set} we give one of the solutions for each equivalence class of $k$ modulo $6$. We conjecture that the pattern continues for all $k$. 

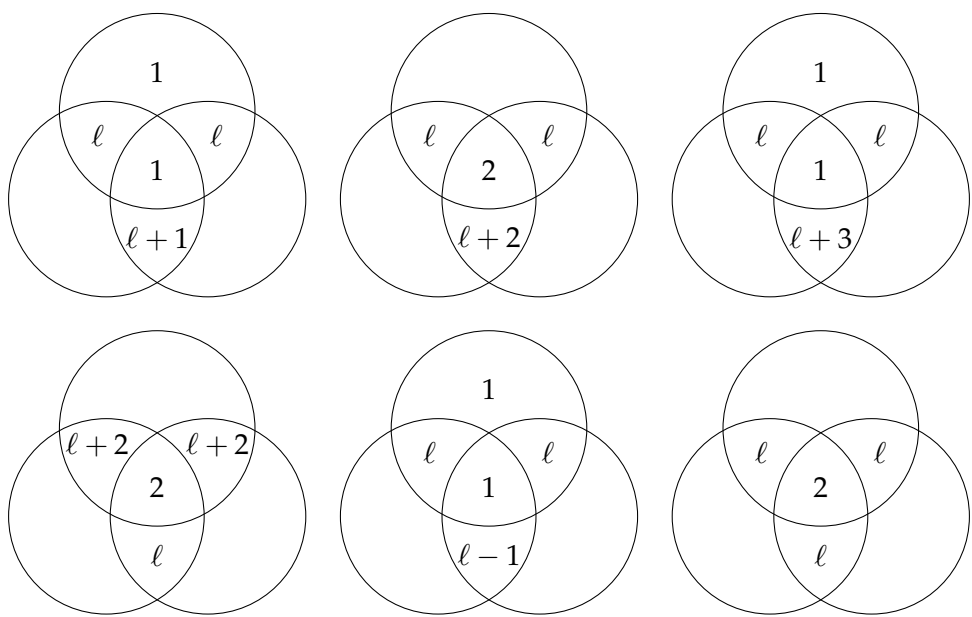
\begin{figure}[H]
  \centering
  \begin{tikzpicture}[inner sep=0,scale=1.3]
    \draw (90:.6) circle (1);
    \draw (210:.6) circle (1);
    \draw (330:.6) circle (1);
    \node at (90:1.9) {};
    \node at (210:1.9) {};
    \node at (330:1.9) {};
    \node at (90:1) (a1) {1};
    \node at (210:1) (a2) {};
    \node at (330:1) (a3) {};
    \node at (-90:.7) (b1) {$\ell+1$};
    \node at (30:.7) (b2) {$\ell$};
    \node at (150:.7) (b3) {$\ell$};
    \node at (0:0) (c) {1};
  \end{tikzpicture}
  \begin{tikzpicture}[inner sep=0,scale=1.3]
    \draw (90:.6) circle (1);
    \draw (210:.6) circle (1);
    \draw (330:.6) circle (1);
    \node at (90:1.9) {};
    \node at (210:1.9) {};
    \node at (330:1.9) {};
    \node at (90:1) (a1) {};
    \node at (210:1) (a2) {};
    \node at (330:1) (a3) {};
    \node at (-90:.7) (b1) {$\ell+2$};
    \node at (30:.7) (b2) {$\ell$};
    \node at (150:.7) (b3) {$\ell$};
    \node at (0:0) (c) {2};
  \end{tikzpicture}
  \begin{tikzpicture}[inner sep=0,scale=1.3]
    \draw (90:.6) circle (1);
    \draw (210:.6) circle (1);
    \draw (330:.6) circle (1);
    \node at (90:1.9) {};
    \node at (210:1.9) {};
    \node at (330:1.9) {};
    \node at (90:1) (a1) {1};
    \node at (210:1) (a2) {};
    \node at (330:1) (a3) {};
    \node at (-90:.7) (b1) {$\ell+3$};
    \node at (30:.7) (b2) {$\ell$};
    \node at (150:.7) (b3) {$\ell$};
    \node at (0:0) (c) {1};
  \end{tikzpicture}
  \begin{tikzpicture}[inner sep=0,scale=1.3]
    \draw (90:.6) circle (1);
    \draw (210:.6) circle (1);
    \draw (330:.6) circle (1);
    \node at (90:1.9) {};
    \node at (210:1.9) {};
    \node at (330:1.9) {};
    \node at (90:1) (a1) {};
    \node at (210:1) (a2) {};
    \node at (330:1) (a3) {};
    \node at (-90:.7) (b1) {$\ell$};
    \node at (35:.75) (b2) {$\ell+2$};
    \node at (145:.75) (b3) {$\ell+2$};
    \node at (0:0) (c) {2};
  \end{tikzpicture}
  \begin{tikzpicture}[inner sep=0,scale=1.3]
    \draw (90:.6) circle (1);
    \draw (210:.6) circle (1);
    \draw (330:.6) circle (1);
    \node at (90:1.9) {};
    \node at (210:1.9) {};
    \node at (330:1.9) {};
    \node at (90:1) (a1) {1};
    \node at (210:1) (a2) {};
    \node at (330:1) (a3) {};
    \node at (-90:.7) (b1) {$\ell-1$};
    \node at (30:.7) (b2) {$\ell$};
    \node at (150:.7) (b3) {$\ell$};
    \node at (0:0) (c) {1};
  \end{tikzpicture}
  \begin{tikzpicture}[inner sep=0,scale=1.3]
    \draw (90:.6) circle (1);
    \draw (210:.6) circle (1);
    \draw (330:.6) circle (1);
    \node at (90:1.9) {};
    \node at (210:1.9) {};
    \node at (330:1.9) {};
    \node at (90:1) (a1) {};
    \node at (210:1) (a2) {};
    \node at (330:1) (a3) {};
    \node at (-90:.7) (b1) {$\ell$};
    \node at (30:.7) (b2) {$\ell$};
    \node at (150:.7) (b3) {$\ell$};
    \node at (0:0) (c) {2};
  \end{tikzpicture}
  \caption{Conjectural $3$-set families with the minimum number of splitters. From left to right, top to bottom, the congruence classes $0,\ldots,5$ of $k$ modulo $6$. (In each case, $\ell$ is an odd integer.)\label{fig:3-set}}
\end{figure}

The pattern here has been verified for $k\leq 60$.

In Figure~\ref{table:3-set}, we give some computational results for the minimum number of splitters of $3$-set families for small values of $k$.
\begin{figure}[H]
\begin{center}
\begin{tabular}{r|c}
	$k$ & $\#$ of splitters in minimum arrangement\\ \hline
	$\leq5$ & 2 \\
	6 & 4 \\
	7 & 6 \\
	8 & 12 \\
	9 & 18 \\
	10 & 36 \\
	11 & 54 \\
	12 & 108 \\
	13 & 180 \\
	14 & 360 \\
	15 & 600 \\
	16 & 1200 \\
	17 & 2000 \\
	18 & 4000 \\
	19 & 7000 \\
	20 & 14000
\end{tabular}
\end{center}
\caption{Minimum number of splitters for $3$-families on $[k]$.\label{table:3-set}}
\end{figure}

\section{Two and three set games}
\label{appendix:two-three}

In this section we catalog the winner and winning strategy for all instances of the splitting game on two or three sets. By Lemma~\ref{reduction lemma}, we may suppose there are $0$ or $1$ elements in each Venn region. We use an empty region to represent $0$ elements and a $\bullet$ to represent $1$ element. Since we will see that each of these games may be won using a (possibly trivial) pairing strategy, the winner is independent of which player goes first. It is also independent of whether or not there are points in the Venn region of multiplicity $0$; we use a $\star$ to suggest either case.

For two set games, there are $8$ arrangements, not considering the multiplicity $0$ region. Of these, $7$ are won by Split and $1$ by Skew. The $8$ arrangements are divided into $5$ types up to symmetry. Of the $5$ types, $4$ are won by Split and $1$ by Skew (see Figure~\ref{fig:2set}). 

\begin{figure}[H]
  \centering
  \begin{tikzpicture}[inner sep=0,scale=.7]
		\draw (0:.6) circle (1);
		\draw (180:.6) circle (1);
		\node at (10:0) (b) {};
		\node at (180:.9) (a1) {};
		\node at (0:.9) (a2) {}; 
		\node at (-90:1.3) (o) {$\star$};     
  \end{tikzpicture}\quad
  \begin{tikzpicture}[inner sep=0,scale=.7]
		\draw (0:.6) circle (1);
		\draw (180:.6) circle (1);
		\node at (10:0) (b) {$\bullet$};
		\node at (180:.9) (a1) {};
		\node at (0:.9) (a2) {};  
		\node at (-90:1.3) (o) {$\star$};    
  \end{tikzpicture}\quad
  \begin{tikzpicture}[inner sep=0,scale=.7]
		\draw (0:.6) circle (1);
		\draw (180:.6) circle (1);
		\node at (10:0) (b) {};
		\node at (180:.9) (a1) {$\bullet$};
		\node at (0:.9) (a2) {};  
		\node at (-90:1.3) (o) {$\star$};    
  \end{tikzpicture}\quad
  \begin{tikzpicture}[inner sep=0,scale=.7]
		\draw (0:.6) circle (1);
		\draw (180:.6) circle (1);
		\node at (10:0) (b) {$\bullet$};
		\node at (180:.9) (a1) {$\bullet$};
		\node at (0:.9) (a2) {};      
		\node at (-90:1.3) (o) {$\star$};
		\draw[<->,blue] (a1)--(b);
  \end{tikzpicture}\qquad    
	\begin{tikzpicture}[inner sep=0,scale=.7]
		\draw (0:.6) circle (1);
		\draw (180:.6) circle (1);
		\node at (10:0) (b) {$\bullet$};
		\node at (180:.9) (a1) {$\bullet$};
		\node at (0:.9) (a2) {$\bullet$};
		\node at (-90:1.3) (o) {$\star$};
		\draw[<->,red] (a1)edge[bend right](a2);
	\end{tikzpicture}
  \caption{Types of two set games, with pairing strategies shown. The first four arrangements are won by Split. The final arrangement is won by Skew.\label{fig:2set}}
\end{figure}
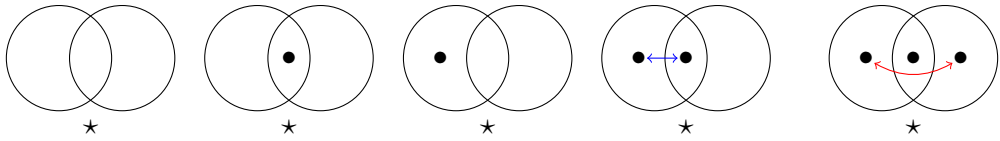

For $3$-set games, there are $128$ arrangements. Of these, $65$ are won by Split and $63$ are won by Skew. The arrangements fall into $40$ types up to symmetry.  By Theorem~\ref{even regions}, Split has a winning strategy for any game with $\leq2$ occurrences of $\bullet$, which accounts for $10$ of the types. The remaining $30$ types are shown below. Figure~\ref{fig:split-winners} depicts games won by Split, and Figure~\ref{fig:skew-winners} depicts games won by Skew.

\begin{figure}[H]
  \centering
  \begin{tikzpicture}[inner sep=0,scale=.7]
    \draw (90:.6) circle (1);
    \draw (210:.6) circle (1);
    \draw (330:.6) circle (1);
    \node at (90:1.9) {};
    \node at (210:1.9) {};
    \node at (330:1.9) {};
    \node at (90:1) (a1) {$\bullet$};
    \node at (210:1) (a2) {$\bullet$};
    \node at (330:1) (a3) {$\bullet$};
    \node at (-90:.7) (b1) {};
    \node at (30:.7) (b2) {};
    \node at (150:.7) (b3) {};
    \node at (0:0) (c) {};
		\node at (-90:1.5) (o) {$\star$};
  \end{tikzpicture}
  \begin{tikzpicture}[inner sep=0,scale=.7]
    \draw (90:.6) circle (1);
    \draw (210:.6) circle (1);
    \draw (330:.6) circle (1);
    \node at (90:1.9) {};
    \node at (210:1.9) {};
    \node at (330:1.9) {};
    \node at (90:1) (a1) {$\bullet$};
    \node at (210:1) (a2) {$\bullet$};
    \node at (330:1) (a3) {};
    \node at (-90:.7) (b1) {};
    \node at (30:.7) (b2) {$\bullet$};
    \node at (150:.7) (b3) {};
    \node at (0:0) (c) {};
		\node at (-90:1.5) (o) {$\star$};
    \draw[<->,blue] (a1)--(b2);
  \end{tikzpicture}
  \begin{tikzpicture}[inner sep=0,scale=.7]
    \draw (90:.6) circle (1);
    \draw (210:.6) circle (1);
    \draw (330:.6) circle (1);
    \node at (90:1.9) {};
    \node at (210:1.9) {};
    \node at (330:1.9) {};
    \node at (90:1) (a1) {$\bullet$};
    \node at (210:1) (a2) {};
    \node at (330:1) (a3) {};
    \node at (-90:.7) (b1) {};
    \node at (30:.7) (b2) {$\bullet$};
    \node at (150:.7) (b3) {};
    \node at (0:0) (c) {$\bullet$};
		\node at (-90:1.5) (o) {$\star$};
    \draw[<->,blue] (c)--(b2);
  \end{tikzpicture}
  \begin{tikzpicture}[inner sep=0,scale=.7]
    \draw (90:.6) circle (1);
    \draw (210:.6) circle (1);
    \draw (330:.6) circle (1);
    \node at (90:1.9) {};
    \node at (210:1.9) {};
    \node at (330:1.9) {};
    \node at (90:1) (a1) {$\bullet$};
    \node at (210:1) (a2) {};
    \node at (330:1) (a3) {};
    \node at (-90:.7) (b1) {};
    \node at (30:.7) (b2) {$\bullet$};
    \node at (150:.7) (b3) {$\bullet$};
    \node at (0:0) (c) {};
		\node at (-90:1.5) (o) {$\star$};
		\draw[<->,blue] (b2)--(b3);
  \end{tikzpicture}
  
  \begin{tikzpicture}[inner sep=0,scale=.7]
    \draw (90:.6) circle (1);
    \draw (210:.6) circle (1);
    \draw (330:.6) circle (1);
    \node at (90:1.9) {};
    \node at (210:1.9) {};
    \node at (330:1.9) {};
    \node at (90:1) (a1) {};
    \node at (210:1) (a2) {};
    \node at (330:1) (a3) {};
    \node at (-90:.7) (b1) {$\bullet$};
    \node at (30:.7) (b2) {$\bullet$};
    \node at (150:.7) (b3) {$\bullet$};
    \node at (0:0) (c) {$\bullet$};
		\node at (-90:1.5) (o) {$\star$};
		\draw[<->,blue] (c)--(b1);
		\draw[<->,blue](b2)--(b3);
  \end{tikzpicture}
  \begin{tikzpicture}[inner sep=0,scale=.7]
    \draw (90:.6) circle (1);
    \draw (210:.6) circle (1);
    \draw (330:.6) circle (1);
    \node at (90:1.9) {};
    \node at (210:1.9) {};
    \node at (330:1.9) {};
    \node at (90:1) (a1) {};
    \node at (210:1) (a2) {};
    \node at (330:1) (a3) {$\bullet$};
    \node at (-90:.7) (b1) {$\bullet$};
    \node at (30:.7) (b2) {};
    \node at (150:.7) (b3) {$\bullet$};
    \node at (0:0) (c) {$\bullet$};
		\node at (-90:1.5) (o) {$\star$};
    \draw[<->,blue](c)--(b3);
		\draw[<->,blue](b1)--(a3);
  \end{tikzpicture}
  \begin{tikzpicture}[inner sep=0,scale=.7]
    \draw (90:.6) circle (1);
    \draw (210:.6) circle (1);
    \draw (330:.6) circle (1);
    \node at (90:1.9) {};
    \node at (210:1.9) {};
    \node at (330:1.9) {};
    \node at (90:1) (a1) {};
    \node at (210:1) (a2) {$\bullet$};
    \node at (330:1) (a3) {$\bullet$};
    \node at (-90:.7) (b1) {$\bullet$};
    \node at (30:.7) (b2) {};
    \node at (150:.7) (b3) {$\bullet$};
    \node at (0:0) (c) {};
		\node at (-90:1.5) (o) {$\star$};
    \draw[<->,blue](a3)--(b1);
		\draw[<->,blue](a2)--(b3);
  \end{tikzpicture}
  \begin{tikzpicture}[inner sep=0,scale=.7]
    \draw (90:.6) circle (1);
    \draw (210:.6) circle (1);
    \draw (330:.6) circle (1);
    \node at (90:1.9) {};
    \node at (210:1.9) {};
    \node at (330:1.9) {};
    \node at (90:1) (a1) {};
    \node at (210:1) (a2) {$\bullet$};
    \node at (330:1) (a3) {$\bullet$};
    \node at (-90:.7) (b1) {$\bullet$};
    \node at (30:.7) (b2) {};
    \node at (150:.7) (b3) {};
    \node at (0:0) (c) {$\bullet$};
		\node at (-90:1.5) (o) {$\star$};
		\draw[<->,blue](c)--(b1);
  \end{tikzpicture}
  
  \begin{tikzpicture}[inner sep=0,scale=.7]
    \draw (90:.6) circle (1);
    \draw (210:.6) circle (1);
    \draw (330:.6) circle (1);
    \node at (90:1.9) {};
    \node at (210:1.9) {};
    \node at (330:1.9) {};
    \node at (90:1) (a1) {};
    \node at (210:1) (a2) {$\bullet$};
    \node at (330:1) (a3) {$\bullet$};
    \node at (-90:.7) (b1) {};
    \node at (30:.7) (b2) {$\bullet$};
    \node at (150:.7) (b3) {$\bullet$};
    \node at (0:0) (c) {$\bullet$};
		\node at (-90:1.5) (o) {$\star$};
    \draw[<->,blue] (a2)--(b3);
    \draw[<->,blue] (c)--(b2);
  \end{tikzpicture}
  \begin{tikzpicture}[inner sep=0,scale=.7]
    \draw (90:.6) circle (1);
    \draw (210:.6) circle (1);
    \draw (330:.6) circle (1);
    \node at (90:1.9) {};
    \node at (210:1.9) {};
    \node at (330:1.9) {};
    \node at (90:1) (a1) {$\bullet$};
    \node at (210:1) (a2) {$\bullet$};
    \node at (330:1) (a3) {$\bullet$};
    \node at (-90:.7) (b1) {$\bullet$};
    \node at (30:.7) (b2) {$\bullet$};
    \node at (150:.7) (b3) {$\bullet$};
    \node at (0:0) (c) {};
		\node at (-90:1.5) (o) {$\star$};
    \draw[<->,blue](a1)--(b3);
    \draw[<->,blue](a2)--(b1);
    \draw[<->,blue](a3)--(b2);
  \end{tikzpicture}
  \caption{Three-set games won by Split with $\geq 3$ instances of $\bullet$, with pairing strategies shown.}
  \label{fig:split-winners}
\end{figure}
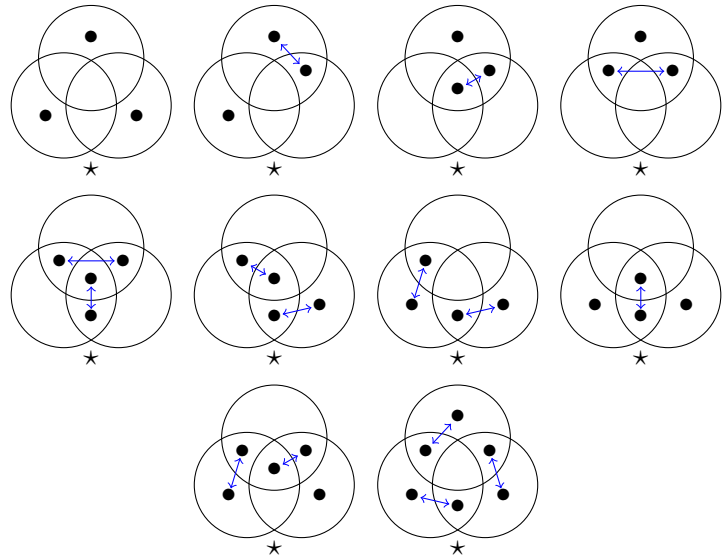
  
\begin{figure}[H]
  \centering
  \begin{tikzpicture}[inner sep=0,scale=.7]
    \draw (90:.6) circle (1);
    \draw (210:.6) circle (1);
    \draw (330:.6) circle (1);
    \node at (90:1.9) {};
    \node at (210:1.9) {};
    \node at (330:1.9) {};
    \node at (90:1) (a1) {};
    \node at (210:1) (a2) {};
    \node at (330:1) (a3) {};
    \node at (-90:.7) (b1) {};
    \node at (30:.7) (b2) {$\bullet$};
    \node at (150:.7) (b3) {$\bullet$};
    \node at (0:0) (c) {$\bullet$};
		\node at (-90:1.5) (o) {$\star$};
		\draw[<->,red](b2)--(b3);
  \end{tikzpicture}
  \begin{tikzpicture}[inner sep=0,scale=.7]
    \draw (90:.6) circle (1);
    \draw (210:.6) circle (1);
    \draw (330:.6) circle (1);
    \node at (90:1.9) {};
    \node at (210:1.9) {};
    \node at (330:1.9) {};
    \node at (90:1) (a1) {};
    \node at (210:1) (a2) {};
    \node at (330:1) (a3) {$\bullet$};
    \node at (-90:.7) (b1) {};
    \node at (30:.7) (b2) {};
    \node at (150:.7) (b3) {$\bullet$};
    \node at (0:0) (c) {$\bullet$};
		\node at (-90:1.5) (o) {$\star$};
		\draw[<->,red](a3)edge[bend left](b3);
  \end{tikzpicture}
  \begin{tikzpicture}[inner sep=0,scale=.7]
    \draw (90:.6) circle (1);
    \draw (210:.6) circle (1);
    \draw (330:.6) circle (1);
    \node at (90:1.9) {};
    \node at (210:1.9) {};
    \node at (330:1.9) {};
    \node at (90:1) (a1) {};
    \node at (210:1) (a2) {$\bullet$};
    \node at (330:1) (a3) {$\bullet$};
    \node at (-90:.7) (b1) {};
    \node at (30:.7) (b2) {};
    \node at (150:.7) (b3) {};
    \node at (0:0) (c) {$\bullet$};
		\node at (-90:1.5) (o) {$\star$};
		\draw[<->,red](a2)--(a3);
  \end{tikzpicture}
  \begin{tikzpicture}[inner sep=0,scale=.7]
    \draw (90:.6) circle (1);
    \draw (210:.6) circle (1);
    \draw (330:.6) circle (1);
    \node at (90:1.9) {};
    \node at (210:1.9) {};
    \node at (330:1.9) {};
    \node at (90:1) (a1) {};
    \node at (210:1) (a2) {};
    \node at (330:1) (a3) {};
    \node at (-90:.7) (b1) {$\bullet$};
    \node at (30:.7) (b2) {$\bullet$};
    \node at (150:.7) (b3) {$\bullet$};
    \node at (0:0) (c) {};
		\node at (-90:1.5) (o) {$\star$};
  \end{tikzpicture}
  \begin{tikzpicture}[inner sep=0,scale=.7]
    \draw (90:.6) circle (1);
    \draw (210:.6) circle (1);
    \draw (330:.6) circle (1);
    \node at (90:1.9) {};
    \node at (210:1.9) {};
    \node at (330:1.9) {};
    \node at (90:1) (a1) {};
    \node at (210:1) (a2) {};
    \node at (330:1) (a3) {$\bullet$};
    \node at (-90:.7) (b1) {$\bullet$};
    \node at (30:.7) (b2) {};
    \node at (150:.7) (b3) {$\bullet$};
    \node at (0:0) (c) {};
		\node at (-90:1.5) (o) {$\star$};
		\draw[<->,red](a3)--(b3);
  \end{tikzpicture}
  \begin{tikzpicture}[inner sep=0,scale=.7]
    \draw (90:.6) circle (1);
    \draw (210:.6) circle (1);
    \draw (330:.6) circle (1);
    \node at (90:1.9) {};
    \node at (210:1.9) {};
    \node at (330:1.9) {};
    \node at (90:1) (a1) {};
    \node at (210:1) (a2) {$\bullet$};
    \node at (330:1) (a3) {$\bullet$};
    \node at (-90:.7) (b1) {$\bullet$};
    \node at (30:.7) (b2) {};
    \node at (150:.7) (b3) {};
    \node at (0:0) (c) {};
		\node at (-90:1.5) (o) {$\star$};
		\draw[<->,red](a2)edge[bend left](a3);
  \end{tikzpicture}
	
  \begin{tikzpicture}[inner sep=0,scale=.7]
    \draw (90:.6) circle (1);
    \draw (210:.6) circle (1);
    \draw (330:.6) circle (1);
    \node at (90:1.9) {};
    \node at (210:1.9) {};
    \node at (330:1.9) {};
    \node at (90:1) (a1) {$\bullet$};
    \node at (210:1) (a2) {};
    \node at (330:1) (a3) {};
    \node at (-90:.7) (b1) {};
    \node at (30:.7) (b2) {$\bullet$};
    \node at (150:.7) (b3) {$\bullet$};
    \node at (0:0) (c) {$\bullet$};
		\node at (-90:1.5) (o) {$\star$};
		\draw[<->,red](b2)--(b3);
  \end{tikzpicture}
  \begin{tikzpicture}[inner sep=0,scale=.7]
    \draw (90:.6) circle (1);
    \draw (210:.6) circle (1);
    \draw (330:.6) circle (1);
    \node at (90:1.9) {};
    \node at (210:1.9) {};
    \node at (330:1.9) {};
    \node at (90:1) (a1) {$\bullet$};
    \node at (210:1) (a2) {};
    \node at (330:1) (a3) {$\bullet$};
    \node at (-90:.7) (b1) {};
    \node at (30:.7) (b2) {};
    \node at (150:.7) (b3) {$\bullet$};
    \node at (0:0) (c) {$\bullet$};
		\node at (-90:1.5) (o) {$\star$};
		\draw[<->,red](a3)edge[bend left](b3);
  \end{tikzpicture}
  \begin{tikzpicture}[inner sep=0,scale=.7]
    \draw (90:.6) circle (1);
    \draw (210:.6) circle (1);
    \draw (330:.6) circle (1);
    \node at (90:1.9) {};
    \node at (210:1.9) {};
    \node at (330:1.9) {};
    \node at (90:1) (a1) {$\bullet$};
    \node at (210:1) (a2) {$\bullet$};
    \node at (330:1) (a3) {$\bullet$};
    \node at (-90:.7) (b1) {};
    \node at (30:.7) (b2) {};
    \node at (150:.7) (b3) {};
    \node at (0:0) (c) {$\bullet$};
		\node at (-90:1.5) (o) {$\star$};
		\draw[<->,red](a2)--(a3);
  \end{tikzpicture}
  \begin{tikzpicture}[inner sep=0,scale=.7]
    \draw (90:.6) circle (1);
    \draw (210:.6) circle (1);
    \draw (330:.6) circle (1);
    \node at (90:1.9) {};
    \node at (210:1.9) {};
    \node at (330:1.9) {};
    \node at (90:1) (a1) {$\bullet$};
    \node at (210:1) (a2) {};
    \node at (330:1) (a3) {};
    \node at (-90:.7) (b1) {$\bullet$};
    \node at (30:.7) (b2) {$\bullet$};
    \node at (150:.7) (b3) {$\bullet$};
    \node at (0:0) (c) {};
		\node at (-90:1.5) (o) {$\star$};
		\draw[<->,red](b2)--(b3);
  \end{tikzpicture}
  \begin{tikzpicture}[inner sep=0,scale=.7]
    \draw (90:.6) circle (1);
    \draw (210:.6) circle (1);
    \draw (330:.6) circle (1);
    \node at (90:1.9) {};
    \node at (210:1.9) {};
    \node at (330:1.9) {};
    \node at (90:1) (a1) {$\bullet$};
    \node at (210:1) (a2) {};
    \node at (330:1) (a3) {$\bullet$};
    \node at (-90:.7) (b1) {$\bullet$};
    \node at (30:.7) (b2) {};
    \node at (150:.7) (b3) {$\bullet$};
    \node at (0:0) (c) {};
		\node at (-90:1.5) (o) {$\star$};
		\draw[<->,red](a3)--(b3);
  \end{tikzpicture}
  \begin{tikzpicture}[inner sep=0,scale=.7]
    \draw (90:.6) circle (1);
    \draw (210:.6) circle (1);
    \draw (330:.6) circle (1);
    \node at (90:1.9) {};
    \node at (210:1.9) {};
    \node at (330:1.9) {};
    \node at (90:1) (a1) {$\bullet$};
    \node at (210:1) (a2) {$\bullet$};
    \node at (330:1) (a3) {$\bullet$};
    \node at (-90:.7) (b1) {$\bullet$};
    \node at (30:.7) (b2) {};
    \node at (150:.7) (b3) {};
    \node at (0:0) (c) {};
		\node at (-90:1.5) (o) {$\star$};
		\draw[<->,red](a2)edge[bend left](a3);
  \end{tikzpicture}
  
  \begin{tikzpicture}[inner sep=0,scale=.7]
    \draw (90:.6) circle (1);
    \draw (210:.6) circle (1);
    \draw (330:.6) circle (1);
    \node at (90:1.9) {};
    \node at (210:1.9) {};
    \node at (330:1.9) {};
    \node at (90:1) (a1) {};
    \node at (210:1) (a2) {};
    \node at (330:1) (a3) {$\bullet$};
    \node at (-90:.7) (b1) {$\bullet$};
    \node at (30:.7) (b2) {$\bullet$};
    \node at (150:.7) (b3) {$\bullet$};
    \node at (0:0) (c) {$\bullet$};
		\node at (-90:1.5) (o) {$\star$};
    \draw[<->,red] (a3)edge[bend right](b3);
    \draw[<->,red] (b1)--(b2);
  \end{tikzpicture}
  \begin{tikzpicture}[inner sep=0,scale=.7]
    \draw (90:.6) circle (1);
    \draw (210:.6) circle (1);
    \draw (330:.6) circle (1);
    \node at (90:1.9) {};
    \node at (210:1.9) {};
    \node at (330:1.9) {};
    \node at (90:1) (a1) {};
    \node at (210:1) (a2) {$\bullet$};
    \node at (330:1) (a3) {$\bullet$};
    \node at (-90:.7) (b1) {$\bullet$};
    \node at (30:.7) (b2) {};
    \node at (150:.7) (b3) {$\bullet$};
    \node at (0:0) (c) {$\bullet$};
		\node at (-90:1.5) (o) {$\star$};
    \draw[<->,red] (a2)--(c);
    \draw[<->,red] (a3)edge[bend right](b3);
  \end{tikzpicture}
  \begin{tikzpicture}[inner sep=0,scale=.7]
    \draw (90:.6) circle (1);
    \draw (210:.6) circle (1);
    \draw (330:.6) circle (1);
    \node at (90:1.9) {};
    \node at (210:1.9) {};
    \node at (330:1.9) {};
    \node at (90:1) (a1) {};
    \node at (210:1) (a2) {$\bullet$};
    \node at (330:1) (a3) {$\bullet$};
    \node at (-90:.7) (b1) {$\bullet$};
    \node at (30:.7) (b2) {$\bullet$};
    \node at (150:.7) (b3) {$\bullet$};
    \node at (0:0) (c) {};
		\node at (-90:1.5) (o) {$\star$};
    \draw[<->,red] (a3)--(b3);
    \draw[<->,red] (a2)--(b2);
  \end{tikzpicture}
  \begin{tikzpicture}[inner sep=0,scale=.7]
    \draw (90:.6) circle (1);
    \draw (210:.6) circle (1);
    \draw (330:.6) circle (1);
    \node at (90:1.9) {};
    \node at (210:1.9) {};
    \node at (330:1.9) {};
    \node at (90:1) (a1) {$\bullet$};
    \node at (210:1) (a2) {$\bullet$};
    \node at (330:1) (a3) {$\bullet$};
    \node at (-90:.7) (b1) {};
    \node at (30:.7) (b2) {};
    \node at (150:.7) (b3) {$\bullet$};
    \node at (0:0) (c) {$\bullet$};
		\node at (-90:1.5) (o) {$\star$};
    \draw[<->,red] (a3)edge[bend right](b3);
    \draw[<->,red] (a2)edge[bend right](a1);
  \end{tikzpicture}
  \begin{tikzpicture}[inner sep=0,scale=.7]
    \draw (90:.6) circle (1);
    \draw (210:.6) circle (1);
    \draw (330:.6) circle (1);
    \node at (90:1.9) {};
    \node at (210:1.9) {};
    \node at (330:1.9) {};
    \node at (90:1) (a1) {$\bullet$};
    \node at (210:1) (a2) {$\bullet$};
    \node at (330:1) (a3) {$\bullet$};
    \node at (-90:.7) (b1) {};
    \node at (30:.7) (b2) {$\bullet$};
    \node at (150:.7) (b3) {$\bullet$};
    \node at (0:0) (c) {};
		\node at (-90:1.5) (o) {$\star$};
    \draw[<->,red] (a2)--(b2);
    \draw[<->,red] (a3)edge[bend left](a1);
  \end{tikzpicture}
  
  \begin{tikzpicture}[inner sep=0,scale=.7]
    \draw (90:.6) circle (1);
    \draw (210:.6) circle (1);
    \draw (330:.6) circle (1);
    \node at (90:1.9) {};
    \node at (210:1.9) {};
    \node at (330:1.9) {};
    \node at (90:1) (a1) {};
    \node at (210:1) (a2) {$\bullet$};
    \node at (330:1) (a3) {$\bullet$};
    \node at (-90:.7) (b1) {$\bullet$};
    \node at (30:.7) (b2) {$\bullet$};
    \node at (150:.7) (b3) {$\bullet$};
    \node at (0:0) (c) {$\bullet$};
		\node at (-90:1.5) (o) {$\star$};
    \draw[<->,red](a2)--(c);
    \draw[<->,red](b1)--(b2);
    \draw[<->,red](a3) edge [bend right](b3);
  \end{tikzpicture}
  \begin{tikzpicture}[inner sep=0,scale=.7]
    \draw (90:.6) circle (1);
    \draw (210:.6) circle (1);
    \draw (330:.6) circle (1);
    \node at (90:1.9) {};
    \node at (210:1.9) {};
    \node at (330:1.9) {};
    \node at (90:1) (a1) {$\bullet$};
    \node at (210:1) (a2) {$\bullet$};
    \node at (330:1) (a3) {$\bullet$};
    \node at (-90:.7) (b1) {};
    \node at (30:.7) (b2) {$\bullet$};
    \node at (150:.7) (b3) {$\bullet$};
    \node at (0:0) (c) {$\bullet$};
		\node at (-90:1.5) (o) {$\star$};
    \draw[<->,red](a1)--(c);
    \draw[<->,red](a2)edge[bend right](b2);
    \draw[<->,red](a3)edge[bend left](b3);
  \end{tikzpicture}
  \qquad
  \begin{tikzpicture}[inner sep=0,scale=.7]
    \draw (90:.6) circle (1);
    \draw (210:.6) circle (1);
    \draw (330:.6) circle (1);
    \node at (90:1.9) {};
    \node at (210:1.9) {};
    \node at (330:1.9) {};
    \node at (90:1) (a1) {$\bullet$};
    \node at (210:1) (a2) {$\bullet$};
    \node at (330:1) (a3) {$\bullet$};
    \node at (-90:.7) (b1) {$\bullet$};
    \node at (30:.7) (b2) {$\bullet$};
    \node at (150:.7) (b3) {$\bullet$};
    \node at (0:0) (c) {$\bullet$};
		\node at (-90:1.5) (o) {$\star$};
    \draw[<->,red](b2)--(b3);
    \draw[<->,red](a1)edge[bend left](b1);
  \end{tikzpicture}
  \caption{Three-set games won by Skew, with pairing strategies shown.}
  \label{fig:skew-winners}
\end{figure}
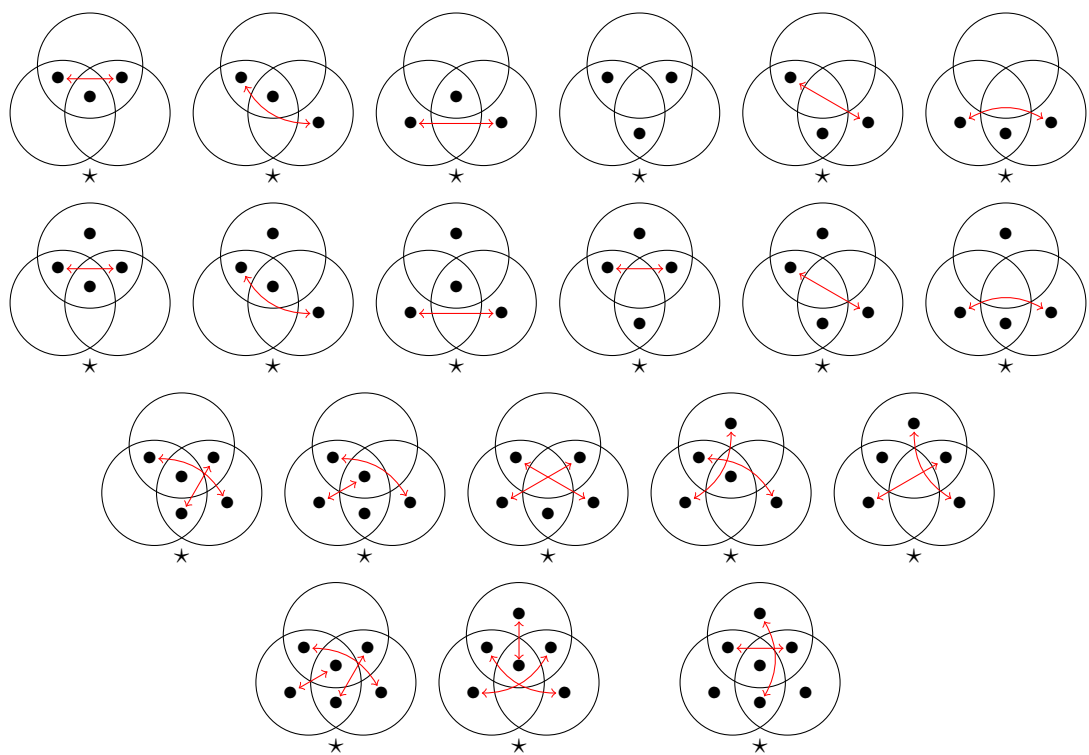

\section{$2$-dimensional splitting tic-tac-toe}
\label{appendix:tictactoe}

In this section we carry out the details of the proof of Theorem~\ref{thm:tictactoe}. Recall the theorem states that in the $m\times n$ splitting tic-tac-toe game with $m,n\neq1$:
\begin{itemize}
  \item If $m=n=3$, Player~II has a winning strategy.
  \item Otherwise, Skew has a winning strategy.
\end{itemize}

For the first item, consider first the case where Skew is Player~I. The board is completely symmetric with respect to the squares, so assume without loss of generality that Skew takes the upper left square. Then Split's strategy is to take the top middle. Up to symmetry, Skew can now take the top right, the left middle, the center, or the right middle.

Suppose Skew takes the top right. Then split takes the left middle. Then if Skew takes the bottom right, then Split takes the right middle. Then as long as Split takes something other than the middle on his remaining turn, Split wins. On the other hand if Skew takes something other than the bottom right on his third turn, then Split takes the bottom right. Now Split has one square from each row and each column, so Skew cannot claim an entire row, and neither can Split, as Split has just one move left. Thus Split wins in this case.

Now suppose Skew takes the left middle. Then Split takes the bottom left. As noted before, if Split has claimed a square from every row and column, Split wins. Thus Skew is forced to take the right middle. Then Split takes the center. Then Split can win by splitting the pair of remaining squares on the right.

Now suppose Skew takes the center. Then Split takes the right middle. Then, so as to prevent Split from taking a square from each row or column, Skew must take the bottom right. Then Split can win by simultaneously splitting the pair of remaining squares on the bottom and the pair of remaining squares on the right.

Finally, consider the case where Skew goes in the right middle. Then Split takes the left middle. Then Skew is forced to take the bottom right, as in the other cases. Then Split takes the top right. Now Split can win by splitting the pair of remaining squares on the bottom.

It remains to show that Skew has a winning strategy in the first item when Split is Player~I. Without loss of generality, Split begins by taking the top left. Then Skew takes the center. Split cannot take a square from the top or the left since Split has the last move, and Skew can force Split to eventually take three in a row. Thus Split has two remaining options, up to symmetry: Split can take the right middle or the bottom right. In either case, after Split makes a move, then Skew takes the bottom middle and forces Split to take the top middle. Then Skew can force Split to eventually take the top right, so Skew wins.

For the second item, if $m,n$ are both even, notice that any row and column forms a two-set family with an odd number of elements in each Venn region. By our previous work on two-set games, Skew has a winning strategy on just these two sets.

If $m=2$, Skew uses a pairing strategy with the pairs $a_{2, j},a_{1, j + 1}$ (cyclically).

For the remaining cases, Skew has a winning strategy by ``getting ahead'' in one of the rows or columns. If at any point, the number of squares claimed by Split and Skew in an even row (column) differ by at least two, or in an odd row (column) by at least three, then Skew can skew the row (column) by claiming squares from that row (column) on each turn until they run out.

If $m$ is even and $n$ is odd, consider first when Skew is Player~I. Skew takes the top left corner. Since $m$ is even, Split cannot allow Skew to take two squares from any column before Split takes any, so Split is forced to go in column 1. Then Skew takes the second square in row 1. Now if Split takes a square from row 1, then Skew takes the second square in column 2, which puts Skew two ahead in column 2. If Split takes a square from column 2, then Skew takes a square from row 1, which puts Skew three ahead in row 1. In either case, Skew wins.

Now consider when $m$ is even and $n$ is odd and Split is Player~I. Without loss of generality, Split goes in the bottom right corner. Then Skew can follow the same strategy as before without Split's first move coming into play at all.

If $m,n \geq 5$ are odd and Skew is Player~I, Skew takes the top left. Assume without loss of generality that Split does not go in column 1 or column 2. Then Skew takes a square from column 1 in an empty row. Then Split is forced to go in column 1. Skew continues taking squares from column 1 in empty rows until Split takes the last square from column 1. Now column 2 is empty, and there are at least two rows in which only Skew has claimed a square. Skew now takes squares from those rows in column 2. Each time, Split is forced to take a square from the same row. Then Skew is free to take a third square from column 2, where Split has taken none, and Skew wins.

If $m,n \geq 5$ are odd and Split is Player~I, assume without loss of generality that Split takes the bottom right square. Then Skew plays as described above, but in the case where Split takes a square not in row 1 in his second move, Skew takes a square from column 1 from the same row, then continues as described. The result is the same, since in either case, it will be Skew's turn with column 2 empty and at least two rows with only Skew having claimed a square.

If $m,n$ are odd and exactly one of them is 3, more explanation is required. Assume $m \geq 5$ and $n = 3$ without loss of generality, and first consider the case where Skew is Player~I. Skew takes the top left. If Split does not go in column 1, then Skew can play as in the $5 \times 5$ game. If Split goes in column 1, then Skew takes the middle square of row 1. Then Split is forced to take the right square of row 1. Then Skew takes squares from column 2 from empty rows until all have been taken, and Split is forced to also take squares from column 2. When Split takes the last square from column 2, there are at least two rows in which only Skew has taken a square. Skew then takes the left square from each of these rows, and Split is forced to take the right square of that row each time. Now Skew takes another square from column 1, which puts Skew three ahead, and Skew wins.

For the last case, assume $m\geq 5$ is odd and $n = 3$ and Split is Player~I. Assume without loss of generality that Split takes the bottom right square. Then Skew takes the top left. If Split doesn't go in go in column 1, then Skew takes squares from column 1, not row $m$, until they run out. Split is forced to take squares from column 1 as well, including the left square in row $m$. Now since there are an odd number of squares, and Split is Player~I, Skew can force Split to eventually take the third square in row $m$, and Skew wins. If Split instead goes in column 1 for his second move, then Skew takes top middle. Split is forced to take top right. Then Skew takes squares from column 2, not row $m$ until they run out, and Split is eventually forced to take bottom middle. Then, since Split has the last move, Skew can force Split to eventually take bottom left, and Skew wins.

This completes the proof.

\bibliographystyle{alpha}
\bibliography{split}

\end{document}